\documentclass[12pt, a4paper, twoside,leqno]{amsart}
\usepackage[centering, totalwidth = 410pt, totalheight = 610pt]{geometry}
\usepackage{amssymb, amsmath, amsthm, listings}
\usepackage{enumerate, microtype, stmaryrd, url} 
\usepackage[latin1]{inputenc}
\usepackage[arrow, matrix, tips, curve, graph]{xy}
 \usepackage{color}
\definecolor{darkgreen}{rgb}{0,0.45,0}
\usepackage[colorlinks,citecolor=darkgreen,linkcolor=darkgreen]{hyperref}
\SelectTips{cm}{10}
\usepackage{mathpartir}

 \DeclareMathOperator{\ob}{ob}

\DeclareMathOperator{\im}{Im}
\newcommand{\cat}[1]{\mathbf{#1}}
\newcommand{\op}{\mathrm{op}}
\newcommand{\id}{\mathrm{id}}
\newcommand{\thg}{{\mathord{\text{--}}}}
\newcommand{\ti}{\!:\!}
\renewcommand{\c}{,\,}
\newcommand{\ty}{\ \mathsf{type}}
\newcommand{\Tm}[2]{\mathrm{Tm}_{#1}(#2)}

\newcommand{\abs}[1]{{\left|{#1}\right|}}
\newcommand{\ab}[1]{{\left[{#1}\right]}}

\newcommand{\res}[2]{\left.{#1}\right|_{#2}}

\newcommand{\defeq}{\mathrel{\mathop:}=}

\newcommand{\cd}[2][]{\vcenter{\hbox{\xymatrix#1{#2}}}}

\renewcommand{\phi}{\varphi}

\newcommand{\C}{{\mathcal C}}
\newcommand{\D}{{\mathcal D}}
\newcommand{\E}{{\mathcal E}}

\renewcommand{\H}{{\mathcal H}}

\newcommand{\J}{{\mathcal J}}

\newcommand{\V}{{\mathcal V}}

\newcommand{\xtor}[1]{\cdl[@1]{{} \ar[r]|-{\object@{|}}^{#1} & {}}}

\makeatletter

\def\hookleftarrowfill@{\arrowfill@\leftarrow\relbar{\relbar\joinrel\rhook}}
\def\twoheadleftarrowfill@{\arrowfill@\twoheadleftarrow\relbar\relbar}
\def\leftbararrowfill@{\arrowdoublefill@{\leftarrow\mkern-5mu}\relbar\mapstochar\relbar\relbar}
\def\Leftbararrowfill@{\arrowdoublefill@{\Leftarrow\mkern-2mu}\Relbar\Mapstochar\Relbar\Relbar}
\def\leftringarrowfill@{\arrowdoublefill@{\leftarrow\mkern-3mu}\relbar{\mkern-3mu\circ\mkern-2mu}\relbar\relbar}
\def\lefttriarrowfill@{\arrowfill@{\mathrel\triangleleft\mkern0.5mu\joinrel\relbar}\relbar\relbar}
\def\Lefttriarrowfill@{\arrowfill@{\mathrel\triangleleft\mkern1mu\joinrel\Relbar}\Relbar\Relbar}

\def\hookrightarrowfill@{\arrowfill@{\lhook\joinrel\relbar}\relbar\rightarrow}
\def\twoheadrightarrowfill@{\arrowfill@\relbar\relbar\twoheadrightarrow}
\def\rightbararrowfill@{\arrowdoublefill@{\relbar\mkern-0.5mu}\relbar\mapstochar\relbar\rightarrow}
\def\Rightbararrowfill@{\arrowdoublefill@{\Relbar\mkern-2mu}\Relbar\Mapstochar\Relbar\Rightarrow}
\def\rightringarrowfill@{\arrowdoublefill@\relbar\relbar{\mkern-2mu\circ\mkern-3mu}\relbar{\mkern-3mu\rightarrow}}
\def\righttriarrowfill@{\arrowfill@\relbar\relbar{\relbar\joinrel\mkern0.5mu\mathrel\triangleright}}
\def\Righttriarrowfill@{\arrowfill@\Relbar\Relbar{\Relbar\joinrel\mkern1mu\mathrel\triangleright}}

\def\leftrightarrowfill@{\arrowfill@\leftarrow\relbar\rightarrow}
\def\mapstofill@{\arrowfill@{\mapstochar\relbar}\relbar\rightarrow}

\renewcommand*\xleftarrow[2][]{\ext@arrow 20{20}0\leftarrowfill@{#1}{#2}}
\providecommand*\xLeftarrow[2][]{\ext@arrow 60{22}0{\Leftarrowfill@}{#1}{#2}}
\providecommand*\xhookleftarrow[2][]{\ext@arrow 10{20}0\hookleftarrowfill@{#1}{#2}}
\providecommand*\xtwoheadleftarrow[2][]{\ext@arrow 60{20}0\twoheadleftarrowfill@{#1}{#2}}
\providecommand*\xleftbararrow[2][]{\ext@arrow 10{22}0\leftbararrowfill@{#1}{#2}}
\providecommand*\xLeftbararrow[2][]{\ext@arrow 50{24}0\Leftbararrowfill@{#1}{#2}}
\providecommand*\xleftringarrow[2][]{\ext@arrow 10{26}0\leftringarrowfill@{#1}{#2}}
\providecommand*\xlefttriarrow[2][]{\ext@arrow 80{24}0\lefttriarrowfill@{#1}{#2}}
\providecommand*\xLefttriarrow[2][]{\ext@arrow 80{24}0\Lefttriarrowfill@{#1}{#2}}

\renewcommand*\xrightarrow[2][]{\ext@arrow 01{20}0\rightarrowfill@{#1}{#2}}
\providecommand*\xRightarrow[2][]{\ext@arrow 04{22}0{\Rightarrowfill@}{#1}{#2}}
\providecommand*\xhookrightarrow[2][]{\ext@arrow 00{20}0\hookrightarrowfill@{#1}{#2}}
\providecommand*\xtwoheadrightarrow[2][]{\ext@arrow 03{20}0\twoheadrightarrowfill@{#1}{#2}}
\providecommand*\xrightbararrow[2][]{\ext@arrow 01{22}0\rightbararrowfill@{#1}{#2}}
\providecommand*\xRightbararrow[2][]{\ext@arrow 04{24}0\Rightbararrowfill@{#1}{#2}}
\providecommand*\xrightringarrow[2][]{\ext@arrow 01{26}0\rightringarrowfill@{#1}{#2}}
\providecommand*\xrighttriarrow[2][]{\ext@arrow 07{24}0\righttriarrowfill@{#1}{#2}}
\providecommand*\xRighttriarrow[2][]{\ext@arrow 07{24}0\Righttriarrowfill@{#1}{#2}}

\providecommand*\xmapsto[2][]{\ext@arrow 01{20}0\mapstofill@{#1}{#2}}
\providecommand*\xleftrightarrow[2][]{\ext@arrow 10{22}0\leftrightarrowfill@{#1}{#2}}
\providecommand*\xLeftrightarrow[2][]{\ext@arrow 10{27}0{\Leftrightarrowfill@}{#1}{#2}}

\makeatother

\newcommand{\twocong}[2][0.5]{\ar@{}[#2] \save ?(#1)*{\cong}\restore}
\newcommand{\twoeq}[2][0.5]{\ar@{}[#2] \save ?(#1)*{=}\restore}
\newcommand{\rtwocell}[3][0.5]{\ar@{}[#2] \ar@{=>}?(#1)+/l 0.2cm/;?(#1)+/r 0.2cm/^{#3}}
\newcommand{\ltwocell}[3][0.5]{\ar@{}[#2] \ar@{=>}?(#1)+/r 0.2cm/;?(#1)+/l 0.2cm/^{#3}}
\newcommand{\ltwocello}[3][0.5]{\ar@{}[#2] \ar@{=>}?(#1)+/r 0.2cm/;?(#1)+/l 0.2cm/_{#3}}
\newcommand{\dtwocell}[3][0.5]{\ar@{}[#2] \ar@{=>}?(#1)+/u  0.2cm/;?(#1)+/d 0.2cm/^{#3}}
\newcommand{\dltwocell}[3][0.5]{\ar@{}[#2] \ar@{=>}?(#1)+/ur  0.2cm/;?(#1)+/dl 0.2cm/^{#3}}
\newcommand{\drtwocell}[3][0.5]{\ar@{}[#2] \ar@{=>}?(#1)+/ul  0.2cm/;?(#1)+/dr 0.2cm/^{#3}}
\newcommand{\dthreecell}[3][0.5]{\ar@{}[#2] \ar@3{->}?(#1)+/u  0.2cm/;?(#1)+/d 0.2cm/^{#3}}
\newcommand{\utwocell}[3][0.5]{\ar@{}[#2] \ar@{=>}?(#1)+/d 0.2cm/;?(#1)+/u 0.2cm/_{#3}}
\newcommand{\dtwocelltarg}[3][0.5]{\ar@{}#2 \ar@{=>}?(#1)+/u  0.2cm/;?(#1)+/d 0.2cm/^{#3}}
\newcommand{\utwocelltarg}[3][0.5]{\ar@{}#2 \ar@{=>}?(#1)+/d  0.2cm/;?(#1)+/u 0.2cm/_{#3}}

\newcommand{\pushoutcorner}[1][dr]{\save*!/#1-1.2pc/#1:(-1,1)@^{|-}\restore}

\newdir{(}{{}*!<0em,-.14em>-\cir<.14em>{l^r}}
\newdir{ (}{{}*!/-5pt/\dir{(}}
\newdir{ >}{{}*!/-5pt/\dir{>}}

\theoremstyle{definition}

\theoremstyle{plain}
\newtheorem{Thm}{Theorem}
\newtheorem{Prop}[Thm]{Proposition}
\newtheorem{Cor}[Thm]{Corollary}

\numberwithin{equation}{section}

\theoremstyle{definition}
\newtheorem{Defn}[Thm]{Definition}

\newtheorem{Ex}[Thm]{Example}

\newtheorem{Rk}[Thm]{Remark}

\DeclareMathOperator{\Inc}{Inc}

\begin{document}
\leftmargini=2em \title{Combinatorial structure of type dependency}
\author{Richard Garner} \address{Department of Mathematics, Macquarie
  University, NSW 2109, Australia} \email{richard.garner@mq.edu.au}
\date{\today}

\thanks{This work was supported by the Australian Research Council's
  \emph{Discovery Projects} scheme, grant number DP110102360.}

\begin{abstract}
  We give an account of the basic combinatorial structure underlying
  the notion of type dependency. We do so by considering the category
  of all dependent sequent calculi, and exhibiting it as the category
  of algebras for a monad on a presheaf category. The objects of the
  presheaf category encode the basic judgements of a dependent sequent
  calculus, while the action of the monad encodes the deduction rules;
  so by giving an explicit description of the monad, we obtain an
  explicit account of the combinatorics of type dependency. We find
  that this combinatorics is controlled by a particular kind of
  decorated ordered tree, familiar from computer science and from
  innocent game semantics.  Furthermore, we find that the monad at
  issue is of a particularly well-behaved kind: it is local right
  adjoint in the sense of Street--Weber. In future work, we will use
  this fact to describe nerves for dependent type theories, and to
  study the coherence problem for dependent type theory using the
  tools of two-dimensional monad theory.
\end{abstract}

\maketitle
\section{Introduction}
There has been much recent interest in Martin-L\"of's type theory,
spurred on by the discovery of remarkable links to algebraic topology
and the theory of $(\infty,1)$-categories.  \emph{Homotopy type
  theory}~\cite{HoTT2013} extends Martin-L\"of type theory with
Voevoedsky's \emph{univalence axiom} and a new collection of
type-formers, the \emph{higher inductive types}; the resultant system
is capable of deriving key results of homotopy theory---such as
calculations of homotopy groups of spheres---in a synthetic, axiomatic
manner. Models of the axioms include not only classical homotopy
theory, but also ``non-standard homotopy theories'' described by
$(\infty, 1)$-toposes (homotopical analogues of categories of
sheaves); in fact, it is believed that we can view homotopy type
theory as providing an internal language for $(\infty, 1)$-toposes,
just as first-order geometric logic does for Grothendieck toposes.

The suitability of Martin-L\"of type theory as a language for abstract
homotopy theory is due to the presence of \emph{identity types} which
can be validly interpreted by the homotopy relation. The existence of
identity types relies in turn on the possibility of \emph{type
  dependency}: type families indexed by elements of other types. While
the intuitive meaning of type dependency is clear, its syntactic
expression is rather involved; a desire to understand its mathematical
essence has led various
authors~\cite{Cartmell1986Generalised,Ehrhard1988Une-semantique,Hyland1989The-theory,Jacobs1993Comprehension,Hofmann1995Extensional,Dybjer1996Internal,Pitts2000Categorical}
to describe notions of \emph{categorical model} for dependent type
theory which abstract away from the complexities of the syntax.

One aspect that remains implicit in both the syntactic and the
categorical accounts is the combinatorial structure of type
dependency: the structure imposed on the judgements of a dependent
sequent calculus by the basic rules of weakening, projection and
substitution.  On the syntactic side, this combinatorics is hidden in
the recursive clauses which generate the calculus; while on the
categorical side, the essential role it plays in constructing models
from the syntax is no longer visible in the finished product. In
short, the syntactic approach fails to detect this structure by being
insufficiently abstract, while the categorical approach fails to see
it by being too abstract.

The objective of this paper is to elucidate the combinatorics of type
dependency by adopting a viewpoint which is intermediate between the
concrete syntactic one and the fully abstract categorical one. We will
model dependent sequent calculi as algebras for a monad on a presheaf
category (we assume the reader is familiar with the basic concepts of
category theory as set out in~\cite{Mac-Lane1971Categories}). Objects
of the presheaf category will encode the basic judgement-forms of a
sequent calculus; the algebraic structure imposed on them by the monad
will encode the deduction rules. Now the combinatorial structure we
wish to describe inheres in the action of the monad, and so by giving
an explicit description of this action, we obtain an explicit account
of the structure. More precisely, the underlying endofunctor of the
monad describes the derivable judgements of a freely-generated sequent
calculus; while the monad multiplication encodes the process of
proof-tree normalisation by which such derivations are combined.

For the dependent sequent calculi to be studied in this paper, we will
not consider rules for type-formers such as $\Pi$-types,
$\Sigma$-types and identity types, but rather concentrate on the core
structural rules of weakening, projection and substitution.
The combinatorics arising from just these rules is particularly
elegant; we will see that in a freely-generated theory of this kind,
the shape of derivable judgements is controlled by suitably decorated
\emph{heaps}. A heap is a finite tree with a total order on its nodes
refining the tree order.  This structure is common throughout logic
and computer science, and the manner in which it appears here is
highly reminiscent of its role in the study of (logical) game
semantics and innocent strategies~\cite{Harmer2007Categorical}. We
hope to explore this link further in future work.

Beyond elucidating a structure which we believe to be interesting in
its own right, the approach taken in this paper will also enable the
analysis of dependent type theory using the tools of
\emph{combinatorial category theory}. This is a particular strand of
category theory, growing out of Joyal's
work~\cite{Joyal1986Foncteurs}, which has found recent
applications~\cite{Leinster2004Operads,Weber2007Familial,Kock2010Polynomial,Weber2013Multitensors}
in taming some of the complexities of higher-dimensional category
theory. A central theme in combinatorial category theory is the study
of monads possessing abstract categorical properties that allow them
to be seen as fundamentally combinatorial in nature. It turns out that
the monad for dependent type theories is of this kind. More precisely,
it is \emph{local right adjoint} or \emph{familially representable} in
the sense of~\cite{Carboni1995Connected,Street2000The-petit,Leinster2004Higher}: and this
permits the application of a rich body of
theory~\cite{Weber2004Generic,Leinster2004Operads,Weber2007Familial,Berger2012Monads}
concerning such monads to the study of dependent type theories.  It is
beyond the scope of this paper to investigate these connections in
detail but let us mention two applications we intend to pursue in
future work; see Section~\ref{sec:final} below for a more detailed
sketch of these applications.

Firstly, we will apply the results of~\cite{Weber2007Familial} to
describe a \emph{nerve functor} for dependent sequent calculi: thus, a
fully faithful embedding of the category of dependent sequent calculi
into a presheaf category, together with a characterisation of the
objects in the image. We hope to use this nerve functor to reveal an
implicit geometry of dependent sequent calculi.  Secondly, we will
provide a new take on the \emph{coherence
  problem}~\cite{Hofmann1995On-the-interpretation} for dependent type
theory. We will do so by lifting the monad for dependent type theories
to a $2$-monad on a presheaf $2$-category, and studying its
pseudoalgebras, which are abstract presentations of Curien's ``syntax
with substitution up to isomorphism''~\cite{Curien1993Substitution}.

We now give an overview of the contents of the paper. We begin in
Section~\ref{sec:syntax} with a description of the syntax of the
dependent sequent calculi of interest to us, which are the
\emph{generalised algebraic theories}
of~\cite{Cartmell1986Generalised}; we also describe the
interpretations between two such theories, so yielding the objects and
morphisms of the category $\cat{GAT}$ of generalised algebraic
theories.

In Section~\ref{sec:termtype}, we begin our combinatorial analysis of
the category $\cat{GAT}$ by describing a presheaf category $[\H^\op,
\cat{Set}]$ whose objects encode the basic judgements of a generalised
algebraic theory.  We define a forgetful functor $\cat{GAT} \to
[\H^\op, \cat{Set}]$, and show that this has a left adjoint and is
monadic. The main goal of the paper will be to give an explicit
description of the induced monad.

In fact it will be convenient---and illuminating---to split this task
up. Writing $w$, $p$, and $s$ for the weakening, projection and
substitution rules, we consider for each $D \subset \{w,p,s\}$ the
category $D\text-\cat{GAT}$ of ``generalised algebraic theories with
structural rules from $D$''. Again, each forgetful functor
$D\text-\cat{GAT} \to [\H^\op, \cat{Set}]$ has a left adjoint and is
monadic, and by studying the induced monads for various choices of
$D$, we may understand the structure induced by various combinations
of the three rules.  

In Section~\ref{sec:weaken}, we consider the structure imposed on
$[\H^\op, \cat{Set}]$ by the rule of weakening alone, without
projection or substitution; thus, we characterise the monad $W$
induced by the forgetful functor $\{w\}\text-\cat{GAT} \to [\H^\op,
\cat{Set}]$.
Then in Section~\ref{sec:project}, we consider the structure imposed
by the projection rule. As this rule in fact relies on the weakening
rule for its well-formedness, we are forced to consider both together:
we thus describe the monad $P$ induced by the functor
$\{w,p\}\text-\cat{GAT} \to [\H^\op, \cat{Set}]$.
In Section~\ref{sec:subst}, we go on to consider the structure imposed by
substitution alone, which involves
describing the monad $S$ induced by the forgetful functor
$\{s\}\text-\cat{GAT} \to [\H^\op, \cat{Set}]$; and then in
Section~\ref{sec:combining}, we combine together the monads for weakening
and projection and for substitution into a compound monad $T = PS$ for
generalised algebraic theories. The extra datum required to do so is a
\emph{distributive law} in the sense of~\cite{Beck1969Distributive}
between the two monads $P$ and $S$; this distributive law describes
the process by which instances of weakening or projection may be
commuted past instances of substitution in a derivation~tree.

We conclude in Section~\ref{sec:final} by showing that each of the
monads constructed in the preceding sections has the good property of
being local right adjoint. We show that the monads $W$, $P$ and $S$
for weakening, weakening and projection, and substitution, have the
additional property of being \emph{cartesian}, meaning that the
naturality squares of their unit and multiplication are all
pullbacks. We also discuss in more detail the further
applications outlined above.

\section{Syntax of type theory}\label{sec:syntax}

\subsection{Generalised algebraic theories}
\label{sec:gener-algebr-theor}
In this section, we give a careful exposition of the syntax of
dependent type theory. As explained in the introduction, our concern
is not with the type constructors of Martin--L\"of type
theory---identity types, $\Pi$-types, $\Sigma$-types, and so on---but
rather with the basic structure of type dependency itself. It is thus
a natural choice to work in the setting of Cartmell's
\emph{generalised algebraic theories}~\cite{Cartmell1986Generalised};
these are dependent sequent calculi without type constructors, but
with the possibility of adding arbitrary (possibly dependent) type and
term constants.  To give such a theory is to give its type and term
constants together with a list of axioms specifying the formation
rules for the constants as well as any equational constraints they
should satisfy.

\begin{Ex}
The \emph{generalised algebraic theory of
  categories} is given over the language with two type-constants $O$
and $A$, two term-constants $\mathsf{c}$ and $\mathsf{i}$, and the
following axioms:
\begin{itemize}
\item $\vdash O\ty$;
\item $x \ti O,\, y \ti O \,\vdash\, A(x,y)\
\mathsf{type}$;
\item $x \ti O \,\vdash\, \mathsf{i}(x) \ti A(x,x) 
$;
\item $x \ti O,\, y \ti O,\, z \ti O,\, g
\ti A(y,z),\, f \ti A(x,y) \,\vdash\, \mathsf{c}(x,y,z,g,f) : A(x,z)$;
\item $x \ti O,\, y \ti O,\, f \ti A(x,y) \vdash \mathsf{c}(x,y,y,\mathsf{i}(y), f)
= f : A(x,y)$;
\item $x \ti O,\, y \ti O,\, f \ti A(x,y) \vdash \mathsf{c}(x,x,y,f,\mathsf{i}(x))
= f : A(x,y)$;
\item $x \ti O,\, y \ti O,\, z \ti O,\, w \ti O,\, h \ti A(z,w),\, g
  \ti A(y,z),\, f \ti A(x,y)$\\ \phantom{a} \hfill $
\vdash\, \mathsf{c}(x,y,w,\mathsf{c}(y,z,w,h,g), f) =
\mathsf{c}(x,z,w,h,\mathsf{c}(x,y,z,g,f)) : A(x,w)$.
\end{itemize}
\end{Ex}
We now give the formal definition;
here, and throughout the paper, $V$ denotes a fixed denumerable set of
variables.

\begin{Defn}\label{def:rawsyntax}
  \begin{enumerate}[(a)]
  \item 
  Given an alphabet $W$, the collection
  $W^*$ of \emph{expressions} over $W$ is the smallest collection
  of strings closed under the rules:
  \begin{itemize}
  \item If $x \in V$ then $x \in W^*$;
  \item If $n \in \mathbb N$, $e_1, \dots, e_n \in W^*$ and $w \in W$ then
    $w(e_1, \dots, e_n) \in W^*$.
  \end{itemize}
  For the second clause, in the case $n = 0$, we abbreviate $w()$
  simply to $w$.  We now define in the usual manner the \emph{free
    variables} $\mathsf{fv}(e)$ of an expression, and the
  \emph{substitution} $e'[e/x]$ of an expression for a variable in an
  expression.\vskip0.5\baselineskip

\item A \emph{context} over the alphabet $W$ is a (possibly empty)
  string $x_1 : T_1, \dots, x_n : T_n$, where each $x_i$ is a distinct
  element of $V$ and each $T_i$ is in $ W^*$. A \emph{judgement} over
  $W$ is a string taking one of the following four forms:
  \begin{itemize}
  \item A \emph{type judgement} $\Gamma \vdash  T \
    \mathsf{type}$;
  \item A \emph{term judgement} $\Gamma \vdash t : T$;
  \item A \emph{type equality judgement} $\Gamma \vdash  T = T' \
    \mathsf{type}$;
  \item A \emph{term equality judgement} $\Gamma \vdash t = t' : T$,
  \end{itemize}
  where in each case $\Gamma$ is a context over $W$ and $t, t', T, T'
  \in W^*$. The \emph{degree} of a judgement is defined to be one
  greater than the length of its context.  \vskip0.5\baselineskip

  \item
  The \emph{boundary} $\partial(\J)$ of a judgement $\J$ over $W$ 
  is a list of judgements of length $0$, $1$, or $2$, defined as follows:
  \begin{itemize}
  \item $\partial(\vdash T\ty) = ()$;
  \item $\partial(\Gamma, x : T \vdash T'\
    \mathsf{type}) = (\Gamma \vdash T\
    \mathsf{type})$;
  \item $\partial(\Gamma \vdash t : T) = (\Gamma \vdash T\
    \mathsf{type})$;
  \item $\partial(\Gamma \vdash T = T'\
    \mathsf{type}) = (\Gamma \vdash T\
    \mathsf{type},\, \Gamma \vdash T'\
    \mathsf{type})$;
  \item $\partial(\Gamma \vdash t = t' : T) = (\Gamma \vdash t : T,\, \Gamma \vdash t' : T)$.
  \end{itemize}\vskip0.5\baselineskip

  \item  A collection $\Phi$ of judgements over $W$ is \emph{deductively
    closed} if, whenever the hypotheses of one the rules in
  Table~\ref{fig1} is in $\Phi$, so too is the conclusion. 
  \end{enumerate}
\end{Defn}

\begin{table}
\small

\textsc{Equality and $\alpha$ rules}
\begin{equation*}
  \inferrule{\Gamma \vdash t : T }{\Gamma \vdash t = t : T} \qquad
  \inferrule{\Gamma \vdash t_1 = t_2 : T }{  \Gamma \vdash t_2 = t_1
    : T} \qquad
  \inferrule{\Gamma \vdash t_1 = t_2 : T\\ \Gamma \vdash t_2 = t_3 :
    T}{\Gamma \vdash t_1 = t_3 : T}
\end{equation*}
\begin{equation*}
  \inferrule{\Gamma \vdash  T \ty}{\Gamma \vdash T = T \ty} \qquad
  \inferrule{\Gamma \vdash T_1 = T_2 \ty}{  \Gamma \vdash T_2 = T_1
    : T} \qquad
  \inferrule{\Gamma \vdash T_1 = T_2 \ty\\ \Gamma \vdash
    T_2 = T_3 \ty}{\Gamma \vdash T_1 = T_3 \ty}
\end{equation*}\begin{equation*}
  \inferrule{\Gamma \vdash T_1 = T_2\ty \\ \Gamma \vdash t :
    T_1}{\Gamma \vdash t : T_2} \qquad
  \inferrule{\Gamma \vdash T_1 = T_2\ty\\ \Gamma \vdash t_1
    = t_2 :
    T_1}{\Gamma \vdash t_1 = t_2 : T_2}
\end{equation*}
\begin{equation*}
  \inferrule*[right=$\sigma
      \in \mathrm{Sym}(V)$]{\Gamma \vdash \J}{\sigma \cdot \Gamma \vdash \sigma \cdot
    \J}
\end{equation*}
(in the last rule, $\J$ denotes one of the four judgment types, and $\sigma \cdot
\Gamma$ and $\sigma \cdot \J$ denote the action of the automorphism
$\sigma$ of $V$ on the variables appearing in $\Gamma$ and $\J$.)
\vskip0.5\baselineskip
\textsc{Weakening rule}
\begin{equation*}
\inferrule*[right={$y \notin \mathsf{fv}(\Gamma) \cup \mathsf{fv}(\Delta)$
    }]{\Gamma \vdash T\ty\\\Gamma, \Delta \vdash \J}{\Gamma, y : T, \Delta
    \vdash \J}
\end{equation*}
\vskip0.5\baselineskip
\textsc{Projection rule}
\begin{equation*}
  \inferrule*[right=$y \notin \mathsf{fv}(\Gamma)$]{\Gamma \vdash T\ty}{\Gamma, y : T
    \vdash y : T}
\end{equation*}
\vskip0.5\baselineskip
\textsc{Substitution rules}
\begin{equation*}
  \inferrule{
  \Gamma \vdash t : T \\ \Gamma, y : T, \Delta \vdash \J}{\Gamma,
  \Delta[t/y] \vdash \J[t / y]}
\end{equation*}
\begin{equation*}
  \inferrule{
  \Gamma \vdash t_1 = t_2: T \\ \Gamma, y : T, \Delta \vdash T'\
  \mathsf{type}}{\Gamma, \Delta[t_2/y]
  \vdash T'[t_1/y] = T'[t_2/y] \ty}
\end{equation*}
\begin{equation*}
  \inferrule{
  \Gamma \vdash t_1 = t_2: T \\ \Gamma, y :
    T, \Delta \vdash t' : T'}{\Gamma,
    \Delta[t_2/y] \vdash t'[t_1/y] = t'[t_2/y] : T'[t_2/y]}
\end{equation*}
\vskip0.5\baselineskip
\caption{Deduction rules for generalised algebraic theories}\label{fig1}
\end{table}

\begin{Defn}~\cite{Cartmell1986Generalised} A \emph{generalised
    algebraic theory} (\textsc{gat}) $\mathbb T$ comprises a
  collection $\Omega$ of type-constants, a collection $\Sigma$ of
  term-constants, and a collection $\Lambda$ of basic judgements over
  $\Omega \amalg \Sigma$ such that:
    \begin{itemize}
 \item For each $A \in \Omega$, there is a unique judgement in $\Lambda$ of
    the form
\[
x_1 : T_1, \dots, x_{n-1} : T_{n-1} \vdash A(x_1, \dots, x_{n-1}) \
\mathsf{type}\rlap{ ,}
\]
and we define the \emph{degree} of $A$ to be the degree of this  judgement;
\item For each $f \in \Sigma$, there is a unique judgement in $\Lambda$ of
  the form
\[
x_1 : T_1, \dots, x_{n-1} : T_{n-1} \vdash f(x_1, \dots, x_{n-1}) : T\rlap{ ,}
\]
and again, we define the \emph{degree} of $f$ to be the degree of this judgement;
\item  All other elements of $\Lambda$ are type equality or term
  equality judgements; 
\item 
  Each element of $\Lambda$ is a derived judgement of $\mathbb T$.
\end{itemize}
Here, the collection $\mathbb T^*$ of derived judgements of $\mathbb
T$ is the smallest deductively closed collection which contains a judgement  $\J
\in \Lambda$ whenever it contains each judgement in its boundary
$\partial(\J)$.
\end{Defn}


As we have said, the notion of \textsc{gat} does not incorporate any
of the usual type-forming operations of Martin-L\"of type theory. To
add these, we would extend the expression grammar of
Definition~\ref{def:rawsyntax} with syntax for the type-formers, and
extend Table~\ref{fig1} with the corresponding formation,
introduction, elimination and computation rules;
see~\cite{Streicher1993Investigations,Hofmann1995Extensional} for
detailed treatments in this spirit. 
In this paper, we are interested in understanding the interactions
between the basic structural rules, and so, as anticipated in the
introduction, we will find it more useful to vary the definition of
\textsc{gat} in the other direction, by \emph{removing} some of the
deduction rules. 

\begin{Defn}
  Let $w$, $p$ and $s$ denote the weakening rules, the projection
  rules and the substitution rules, respectively. For any subset $D
  \subset \{w,p,s\}$, we define an ``$D$-\textsc{gat}'' in the same
  manner as a \textsc{gat}, but with the deduction rules of
  Table~\ref{fig1} reduced to those for equality, $\alpha$-conversion
  and the rules in $D$.  Thus a \textsc{gat} in the previous sense is
  equally a ``$\{w,p,s\}$-\textsc{gat}''.
\end{Defn}

In fact, we should restrict this definition slightly. We call $D
\subset \{w,p,s\}$ \emph{decent} if it contains $w$ whenever it
contains $p$; and in what follows, we consider only decent $D$. The
reason for this restriction is that the projection rule requires the
weakening rule to ``make sense''; more precisely, the problem is that
for indecent $D$, the $D$-\textsc{gats} do not have the property that
the boundary of a derivable judgement is again derivable. We exclude
such pathologies by excluding such $D$.

\subsection{Interpretations}
\label{sec:interpretations}
The ($D$-)\textsc{gat}s are the objects of a category, wherein
morphisms are (equivalence-classes of) \emph{interpretations} of one
theory in another. In terms of this, we can, for example, define
set-based models of a \textsc{gat} $\mathbb T$ as interpretations of
$\mathbb T$ in the ``\textsc{gat} of sets and families of sets'', as
described in~\cite[Section~14]{Cartmell1986Generalised}.

\begin{Defn}\label{def:interpretation}\cite{Cartmell1986Generalised}
  An \emph{interpretation} $\phi \colon \mathbb T \to \mathbb U$ of
  ($D$-)\textsc{gat}s is given by an assignation $\phi \colon \mathbb
  T^* \to \mathbb U^*$ on derived judgements such that:
  \begin{itemize}
  \item $\phi$ preserves boundaries; thus
    $\phi(\partial(\J)) = \partial(\phi(\J))$ for all $\J \in
    \mathbb T^*$.
  \item $\phi$ preserves deduction in the following sense. To each
    deduction rule with $n$ premisses, we can associate an $n$-ary
    partial function $h$ from the judgements over a given alphabet to
    itself, sending the hypotheses of the rule to its conclusion. We
    now require that for each $\J_1, \dots, \J_n \in \mathbb T^*$ we
    have $\phi(h(\J_1, \dots, \J_n)) = h(\phi(\J_1), \dots,
    \phi(\J_n))$ whenever both sides are defined.  For example,
    preserving the first equality rule means that:
  \begin{align*}
  \phi(\Gamma \vdash T\ty) = (\Gamma' \vdash T'\
  \mathsf{type})  & \ \Longrightarrow \ \phi(\Gamma
  \vdash T = T\ty) = (\Gamma'  \vdash T' = T'\
  \mathsf{type})
\end{align*}
while preserving the first substitution rule means that
\begin{multline*}
  \phi(\Gamma \vdash t : T) = (\Gamma' \vdash t' : T'), \phi(\Gamma, y
  : T, \Delta \vdash \J) = (\Gamma', y : T', \Delta' \vdash \J') \\
  \ \Longrightarrow\ \phi(\Gamma,
  \Delta[t/y] \vdash \J[t / y]) = (\Gamma',
  \Delta'[t'/y] \vdash \J'[t' / y])\rlap{ .}
\end{multline*}
\end{itemize}
\end{Defn}

Note that the requirement that an interpretation $\phi$ preserve
deductions means that it is uniquely determined by its action on basic
judgements. More precisely, to specify an interpretation $\phi$ it
suffices to describe where each basic judgement $\J$ of $\mathbb T$ is
sent, and to verify that these choices satisfy $\partial(\phi(\J)) =
\phi(\partial(\J))$; note that, since $\partial(\J)$ is in general
only a derived judgement of $\mathbb T$, the value
$\phi(\partial(\J))$ must be determined from the given values on basic
judgements using the fact that $\phi$ is required to preserve
deduction.

As anticipated above, the morphisms of the category of \textsc{gat}s
will not be interpretations, but rather equivalence classes of
interpretations modulo the equivalence relation of provable equality,
which we now define.

\begin{Defn}
  Let $\mathbb T$ be a ($D$-)\textsc{gat}. The congruence
  $\equiv$ on the derived type judgements of $\mathbb T$ is
  defined by asserting that
\[
  (x_1 : S_1, \dots, x_{n-1} : S_{n-1} \vdash S_{n}\ty) \equiv
    (y_1 : T_1, \dots, y_{m-1} : T_{m-1} \vdash T_{m}\ty)\]
if and only if $n = m$, and for each $1 \leqslant i \leqslant n$,
the judgement
\[
  x_1 : S_1, \dots, x_{i-1} : S_{i-1} \vdash S_i = T_i[x_1/y_1,
  \dots, x_{i-1}/y_{i-1}]\ty
\]
is derivable. The congruence $\equiv$ on the derived term judgements
of $\mathbb T$ is defined by asserting that 
\[
(x_1 : S_1, \dots, x_{n-1} : S_{n-1} \vdash s : S_{n}) \equiv (y_1 :
T_1, \dots, y_{m-1} : T_{m-1} \vdash t : T_{m})\] if and only if their
boundaries are congruent type judgements (in particular $n = m$) and
moreover $\mathbb T$ derives that
\[
  x_1 : S_1, \dots, x_{n-1} : S_{n-1} \vdash s = t[x_1/y_1,
  \dots, x_{n-1}/y_{n-1}] : S_{n}\rlap{ .}
\]
\end{Defn}
\begin{Defn}
  The category $\cat{GAT}$ has generalised algebraic theories as
  objects, and as morphisms, equivalence classes of interpretations
  $\mathbb T \to \mathbb U$, where two interpretations $\phi$ and
  $\phi'$ are deemed equivalent just when $\phi(\J) \equiv \phi'(\J)$
  for each derived type or term judgement of $\mathbb T$. We similarly
  define the category $D\text-\cat{GAT}$ for any decent $D \subset
  \{w,p,s\}$.
\end{Defn}

\section{Type-and-term structures and monadicity}\label{sec:termtype}
\subsection{Type-and-term structures}
We now begin the main task of this paper, that of expressing the
category $\cat{GAT}$ of generalised algebraic theories as the category
of algebras for a monad on a presheaf category. As
discussed in the introduction, the presheaf category at issue will
model the collections of derived judgements of a type theory; more precisely,
it will model the derivable type and term judgements
considered modulo derivable equality.
\begin{Defn}
Let  $\H$ denote the category generated by the graph
\[
\cd{
& 1_t &  2_t & 3_t\\
1 \ar[ur] \ar[r]^{} & 2 \ar[r]^{} \ar[ur]^{}
& 3 \ar[r] \ar[ur]^{} & \cdots\rlap{ .}}
\]
By a \emph{type-and-term structure}, we mean a presheaf $X \in
[\H^\op, \cat{Set}]$. We write the reindexing maps
$X(n_t) \to X(n)$ and $X(n+1) \to X(n)$ as $\partial$, and call
them \emph{boundary maps}. We refer to elements of $X(n)$ as
\emph{type-elements of degree $n$}, and elements of $X(n_t)$ as
\emph{term-elements of degree $n$}.
\end{Defn}

\begin{Rk}\label{rk:game1}
  A type-and-term structure is exactly a \emph{computational arena} in
  the sense of~\cite{Hyland2000On-full}. At the moment, this may
  appear to be a rather fanciful observation, but we will see in
  Remarks~\ref{rk:game2} and~\ref{rk:game3} below that it is part of a
  more substantial link with innocent game semantics.
\end{Rk}

The idea is that, for a type-and-term structure $X$, elements of
$X(n)$ or of $X(n_t)$ should be thought of as $\equiv$-equivalence
classes of type or term judgements of degree $n$ of a dependent
sequent calcluls. The following definition makes this precise.
\begin{Defn}
  Let $D \subset \{w,p,s\}$ be decent. We define the forgetful functor
  $V = V_D \colon D\text-\cat{GAT} \to [\H^\op, \cat{Set}]$ by sending an
  $A$-\textsc{gat} $\mathbb T$ to the type-and-term structure
  $V\mathbb T$ with
\begin{align*}
V\mathbb T(n) &= \{\, (x_1 : T_1, \dots, x_{n-1} : T_{n-1} \vdash T_n\
\mathsf{type}) \in \mathbb T^* \,\} / \mathord \equiv\\
\text{and} \ \ V\mathbb T(n_t) &= \{\,(x_1 : T_1, \dots, x_{n-1} :
T_{n-1} \vdash t : T_n) \in \mathbb T^*\,\}/ \mathord \equiv
\end{align*}
and with the maps $\partial \colon V \mathbb T(n+1) \to V\mathbb T(n)$
and $\partial \colon V\mathbb T(n_t) \to V\mathbb T(n)$
sending the equivalence class of a judgement to the equivalence class
of its boundary.  On maps, $V$ sends an interpretation $\phi
\colon \mathbb T \to \mathbb U$ to the presheaf map $V \phi \colon V
\mathbb T \to V \mathbb U$ with $V\phi([\J]) = [\phi(\J)]$ (note this
is well-defined as an interpretation must preserve boundaries and degrees).
\end{Defn}

In the analysis that follows, we will frequently find that most of the
real action goes on at the level of type-elements, with the
term-elements ``coming along for the ride'' in a fairly
straightforward manner. In light of this, we will find convenient
to introduce the following notation.

\begin{Defn}
  Given $X \in [\H^\op, \cat{Set}]$ and $A \in X(n)$, we write $\Tm X
  A$ for the set of all $a \in X(n_t)$ with $\partial a = A$; given
  $f \colon X \to Y$ in $[\H^\op, \cat{Set}]$, we write $f
  \colon \Tm X {A} \to \Tm Y {f(A)}$ for the restriction of $f$ to
  such term-elements.
\end{Defn}
Now to specify $X \in [\H^\op, \cat{Set}]$, it is enough to give sets
of type-elements and boundary maps $X(1) \leftarrow X(2) \leftarrow
X(3) \leftarrow \dots$ together with sets $\Tm X A$ of term-elements
over each type-element $A$. Similarly, to given a map of presheaves $f
\colon X \to Y$, it is enough to give maps $X(n) \to Y(n)$ for each $n
> 0$ and and maps $\Tm{X}{A} \to \Tm{Y}{f(A)}$ for each $A \in X(n)$.


\subsection{Monadicity}
The following result now tells us that for each decent $D$, we can
present $D\text-\cat{GAT}$ (to within equivalence) as the category of
algebras for a monad on $[\H^\op, \cat{Set}]$. 
\begin{Prop}\label{prop:leftadjmonadic}
  For any decent $D$, the functor $V = V_D \colon D\text-\cat{GAT}
  \to [\H^\op, \cat{Set}]$ has a left adjoint and is monadic.
\end{Prop}
\begin{proof}
  The left adjoint $G = G_D$ has value at a type-and-term structure
  $X$ given by the $D$-\textsc{gat} $GX$ whose type-constants and
  term-constants are the respective type-elements and term-elements of
  $X$, and whose basic judgements are of the form
\[
\J_A = x_1 \ti A_1,\, \dots,\, x_{n-1} \ti
    A_{n-1}(x_1, \dots, x_{n-2})
    \vdash  A(x_1, \dots, x_{n-1}) \ty
\]
for each $A \in X(n)$ with successive boundaries $A_{n-1}, \dots,
A_1$, and
\[
\J_a = x_1 \ti A_1,\, \dots,\, x_{n-1} \ti A_{n-1}(x_1, \dots, x_{n-2})
\vdash a(x_1, \dots, x_{n-1}) \ti A(x_1, \dots, x_{n-1})
\]
for each $a \in \Tm X A$.
The unit of the adjunction $\eta_X \colon X \to VGX$ is given by
$\eta_X(A) = [\J_A]$ and $\eta_X(f) = [\J_f]$; the counit
$\epsilon_\mathbb T \colon GV\mathbb T \to \mathbb T$ at an
$D$-\textsc{gat} $\mathbb T$ is the interpretation defined on basic
judgements by $\epsilon_X(\J_{[\J]}) = \J$.

The monadicity of $V$ is verified by an application of Beck's
theorem~\cite{Mac-Lane1971Categories}. First note that $D\text-\cat{GAT}$ has all
coequalisers: indeed, given interpretations $\phi, \psi \colon \mathbb T
\rightrightarrows \mathbb U$, their coequaliser $\mathbb U'$ is
obtained from $\mathbb U$ by adjoining a basic equality judgement
$\Gamma \vdash T = T'\ty$ whenever $\phi(\J) = (\Gamma \vdash T\ty)$
and $\psi(\J) = (\Gamma' \vdash T'\ty)$ for some basic type judgement
$\J$ of $\mathbb T$; and similarly adjoining a basic equality
judgement $\Gamma \vdash t = t' : T$ whenever $\phi(\J) = (\Gamma
\vdash t : T)$ and $\psi(\J) = (\Gamma' \vdash t' : T')$ for some
basic term judgement $\J$ of $\mathbb T$. The evident interpretation
$\mathbb U \to \mathbb U'$ exhibits $\mathbb U'$ as the coequaliser of
$\phi$ and $\psi$.  

It is easy to see that $V$ reflects isomorphisms, and so to verify
monadicity, it remains to show that $V$ preserves coequalisers of
$V$-split coequaliser pairs.
Let $\phi, \psi \colon \mathbb T \rightrightarrows
\mathbb U$ be interpretations, and let
\[
\cd
{
V\mathbb T \ar@<3pt>[r]^{V\phi}
\ar@<-3pt>[r]_{V\psi} & V\mathbb U \ar[r]^{p}\ar@/^18pt/[l]^\ell & Z \ar@/^8pt/[l]^m
}
\]
be a split coequaliser diagram in $[\H^\op, \cat{Set}]$: thus $pm =
1$, $mp = V\psi.\ell$ and $1 = V\phi.\ell$, and these equations force
$p$ to coequalise $V\phi$ and $V\psi$. We must show that the
coequaliser $q \colon \mathbb U \to \mathbb U'$ of $\phi$ and $\psi$
in $\cat{GAT}$ is preserved by $V$, i.e., that the comparison map
$Vq.m \colon Z \to V\mathbb U'$ in $[\H^\op, \cat{Set}]$ is
invertible.  From the above explicit description of coequalisers in
$D\text-\cat{GAT}$, it is easy to see that the sets comprising
$V\mathbb U'$ are obtained by quotienting out those comprising
$V\mathbb U$ by the smallest equivalence relation $\sim$ such that:
\begin{enumerate}[(a)]
\item $\phi(\J) \sim \psi(\J)$ for each basic type or term judgement
  of $\mathbb T$;
\item $\J_{11} \sim \J_{12}$, \dots, $\J_{k1} \sim \J_{k2}$ implies
  $\J_1 \sim \J_2$ whenever one of the rules of Table~\ref{fig1}
  derives $\J_i$ from $\J_{1i}, \dots, \J_{ki}$ (for $i = 1,2$).
\end{enumerate}
To show that the comparison map $Z \to V\mathbb U'$ is invertible is
equally to show that $a \sim b$ implies $p(a) = p(b)$ for all $a,b$ in
$V\mathbb U$. Clearly $p(\phi(\J)) = p(\psi(\J))$ for all basic type
or term judgements of $\mathbb T$ since $p$ coequalises $V\phi$ and
$V\psi$; it remains to show that if 
that if $p(\J_{j1}) = p(\J_{j2})$ in $V \mathbb U$ (for $j = 1,
\dots, k$) and one of the rules of Table~\ref{fig1} derives $\J_i$
from $\J_{1i}, \dots, \J_{ki}$ (for $i = 1,2$), then $p(\J_1) =
p(\J_2)$. As in Definition~\ref{def:interpretation}, write $h$ for
the $k$-ary partial function on judgements associated to the
derivation rule at issue; thus we have $h(\J_{1i}, \dots, \J_{ki}) =
\J_i$ (for $i = 1,2$). Examining the form of the rules in
Table~\ref{fig1}, we see that definedness of $h$ depends only on
conditions involving boundaries; since $\ell \colon V \mathbb U \to V
\mathbb T$ preserves boundaries, we conclude that $h(\ell(\J_{1i}),
\dots, \ell(\J_{ki}))$ is defined for $i = 1,2$; and now
\begin{align*}
p(\J_i) & = p(h(\J_{1i}, \dots, \J_{ki}))
= p(h(\phi \ell(\J_{1i}),\dots, \phi\ell(\J_{ki})))\\
&= p(\phi(h(\ell(\J_{1i}), \dots, \ell(\J_{ki}))) 
 = p(\psi(h(\ell(\J_{1i}), \dots, \ell(\J_{ki}))) \\
&= p(h(\psi \ell(\J_{1i}),\dots, \psi\ell(\J_{ki}))
= p(h(mp(\J_{1i}),\dots, mp(\J_{ki})))
\end{align*}
for $i = 1,2$. But since $p(\J_{j1}) = p(\J_{j2})$ for $j = 1, \dots,
k$, we conclude that $p(\J_1) = p(\J_2)$ as required.
\end{proof}
When $D = \emptyset$, we have a considerably stronger result: $V_\emptyset$ is
an equivalence. This explains why have chosen to view \textsc{gat}s as
monadic over $[\H^\op, \cat{Set}]$, rather than over some other presheaf
category.

\begin{Prop}\label{prop:equiv1}
  $V_\emptyset \colon \emptyset\text-\cat{GAT} \to [\H^\op, \cat{Set}]$ is an
  equivalence of categories.
\end{Prop}
\begin{proof}
  We already know that $V_\emptyset$ has a left adjoint $G_\emptyset$ and is monadic, so
  it suffices to check that the unit $\eta \colon 1 \Rightarrow
  V_\emptyset G_\emptyset$ of
  the induced monad is invertible; thus, that for any $X \in [\H^\op,
  \cat{Set}]$, each derived type or term judgement of $G_\emptyset X$ is
  $\equiv$-equivalent to a unique basic one.  This
  follows by an easy induction on derivations.
\end{proof}
In the following sections, we describe the monads $V_D G_D$ on
$[\H^\op, \cat{Set}]$ for various decent $A \subset \{w,p,s\}$. Our
eventual objective is to do so in the case $A = \{w,p,s\}$, which we
do in Section~\ref{sec:combining} below. We do this by way of various
simpler cases: $A = \{w\}$ in Section~\ref{sec:weaken}, $A = \{w,p\}$
in Section~\ref{sec:project}, and $A = \{s\}$ in
Section~\ref{sec:subst}. For the moment, let us record those
aspects of the analysis common to all cases.

\begin{Prop}\label{prop:commonanalysis}
  For any decent $A \subset \{w,p,s\}$ and any $X \in [\H^\op,
  \cat{Set}]$, the only derivable type-equalities or term-equalities
  of the free $D$-\textsc{gat} on $X$ are reflexivity judgements
  $\Gamma \vdash T = T \ty$ or $\Gamma \vdash t = t : T$. It follows
  that type- or term-elements of $V_DG_D(X)$ are $\alpha$-equivalence
  classes of derivable judgements in the free $D$-\textsc{gat} on $X$.
\end{Prop}
\begin{proof}
By induction on derivations.
\end{proof}

\section{Weakening}\label{sec:weaken}
In this section, we describe the additional structure imposed on a
type-and-term structure by the weakening rules; in other words, we
will characterise the \emph{weakening monad} $W$ induced by the
free-forgetful adjunction $\{w\}\text-\cat{GAT} \leftrightarrows
[\H^\op, \cat{Set}]$.

\subsection{Underlying endofunctor}
Our first step will be to describe the
underlying endofunctor of the weakening monad. Its basic combinatorics are
controlled by the notion of \emph{min-heap}. Here, and throughout the rest of the paper, we
use the notation $[a, b]$ to indicate the set of natural numbers
$\{a, a+1, \dots, b-1, b\}$, and write $[n]$ to mean $[1, n]$.

\begin{Defn}
  A \emph{min-heap} of size $n$ is a function $\phi \colon [0,n] \to
  [0,n]$ such that $\phi(0) = 0$ and $\phi(i) < i$ for all $i \in
  [n]$.  We write $\mathrm{Hp}(n)$ for the set of min-heaps of size 
  $n$. For $\phi \in \mathrm{Hp}(n+1)$, we write $\partial(\phi)$ for
  $\res{\phi}{[0,n]} \in \mathrm{Hp}(n)$. For $\phi \in
  \mathrm{Hp}(n)$ and $i \in [n]$, the \emph{depth} of $i$ in $\phi$
  is defined to be $\mathsf{dp}_\phi(i) = \abs{\{\phi(i), \phi^2(i),
    \dots\}}$.
\end{Defn}
\newcommand{\hp}{\preccurlyeq}
\newcommand{\dn}{\mathord \downarrow}

To give $\phi \in \mathrm{Hp}(n)$ is equally to give a partial order
$\hp_\phi$ on $[n]$ that is contained in the natural ordering (so $i
\hp_\phi j$ implies $i \leqslant j$) and such that each downset $\dn i
= \{ j : j \hp_\phi i \}$ is a linear order.  The partial order
corresponding to a min-heap $\phi$ is given by $i \hp_\phi j$ iff $i =
\phi^k(j)$ for some $k$; conversely, the function associated to a
partial order $\hp$ is given by $\phi(j) = \max \{ i \in [0,j-1] : i =
0 \text{ or } i \hp j\}$. In what follows, we will more frequently use the
functional representation of min-heaps, but will switch
to the relational representation where this is more convenient.

We may depict $\phi \in \mathrm{Hp}(n)$ by drawing the Hasse diagram
of $\hp_\phi$, which is a non-plane, directed forest with nodes $\{1,
\dots, n\}$, wherein the values labelling the nodes decrease along any
directed path.  For example:
\[
\begin{aligned}
\phi(1) &= 0 & \phi(5) &= 2\\
\phi(2) &= 1 & \phi(6) &= 3\\
\phi(3) &= 0 & \phi(7) &= 1\\
\phi(4) &= 2 & \phi(8) &= 6
\end{aligned} \qquad \longleftrightarrow \qquad
\cd[@-1em]{
4 \ar[dr] & & 5 \ar[dl] & 8 \ar[d] \\ &
2 \ar[d] & 7 \ar[dl] & 6 \ar[d]\\ & 
1 & & 3\rlap{ .}
}
\]
A forest of this kind with $n$ nodes can be seen as specifying the
shape of a type judgement of degree $n$ in a free
$\{w\}$-\textsc{gat}; the numbers describe the ordering of types in
the context, and the arrows indicate the dependencies between
them. The following definition makes this precise.

\begin{Defn}
Given $\phi \in \mathrm{Hp}(n)$, we define the presheaf $[\phi] \in
  [\H^\op, \cat{Set}]$ encoding a type judgement of shape $\phi$ by:
\begin{align*}
\ab\phi(m) &= \{ i \in [ n] : \mathsf{dp}_\phi(i) = m\} & \ab\phi(m_t) = \emptyset
\end{align*}
with the non-trivial boundary maps $\partial \colon \ab\phi(m+1) \to
\ab\phi(m)$ given by $i \mapsto \phi(i)$. 
\end{Defn}
So for any $X \in [\H^\op, \cat{Set}]$, to give a map $h \colon \ab
\phi \to X$ in $[\H^\op, \cat{Set}]$ is to give type-elements $h(1),
\dots, h(n)$ of $X$ such that each $h(i)$ is of degree
$\mathsf{dp}_\phi(i)$, and such that $\partial(h(i)) = h(j)$ whenever
$\phi(i) = j$. If $n > 1$, we write $\partial h$ for the restriction
of $h$ along the evident inclusion $\ab{\partial \phi} \to \ab \phi$.

\begin{Prop}\label{prop:wkendo}
  The value at $X \in [\H^\op, \cat{Set}]$ of the underlying
  endofunctor of the weakening monad has type-elements and boundaries
  given by
   \begin{gather*}
  WX(n) = \sum_{\phi \in \mathrm{Hp}(n)}\, [\H^\op, \cat{Set}](\ab
  \phi, X)  \qquad
\begin{aligned}[t]
  \partial \colon
WX(n+1) & \to WX(n) \\ 
 (\phi, h) & \mapsto (\partial \phi, \partial h)
\end{aligned}
\end{gather*}
and term-elements $\Tm {WX} {\phi, h} = \Tm X {h(n)}$ for each $(\phi,
h) \in WX(n)$.
\end{Prop}
\begin{proof}
By Proposition~\ref{prop:commonanalysis}, type-elements
  of $WX$ are $\alpha$-equivalence classes of type judgements
  in the free $\{w\}$-\textsc{gat} on $X$.  Each
  such equivalence class contains a unique judgement of the form
\begin{equation}\label{eq:jweaken}
\J = x_1 \ti T_1,\, \dots,\, x_{n-1} \ti T_{n-1} \vdash T_n\ty\rlap{ .}
\end{equation}
Now by induction on derivations, we may show that each $T_i$ is of the
form $A_i(x_{j_1}, \dots, x_{j_{k-1}})$ for some $A_i \in X(k)$ and $0
< j_1< \dots < j_{k-1} < i$. For each $i \in [n]$, we define $h(i) =
A_i$, and $\phi(i) = 0$ if $k = 1$ and $\phi(i) = j_{k-1}$
otherwise. Taking also $\phi(0) = 0$, we obtain a min-heap $\phi \in
\mathrm{Hp}(n)$, and by a further induction on derivations, we see
that
\begin{equation}\label{eq:welldef}
\text{if $T_i$ is $A(x_{j_1}, \dots, x_{j_{k-1}})$ and $k > 1$, then
  $T_{j_{k-1}}$ is $(\partial A)(x_{j_1}, \dots, x_{j_{k-2}})$,}
\end{equation}
which  implies that $h$ is a well-defined map $ \ab{\phi}
\to X$. We thus have a function
\begin{equation}
\label{eq:mappingweak}
\begin{aligned}
\theta \colon WX(n) & \to \textstyle \sum_{\phi \in \mathrm{Hp}(n)} [\H^\op,
\cat{Set}](\ab \phi, X)\\
[\J] & \mapsto (\phi_\J, h_\J)
\end{aligned}
\end{equation}
Now \eqref{eq:welldef} ensures that $\theta$ is injective, and it is
clear that $\theta \partial = \partial \theta$. It remains to show
surjectivity of $\theta$.  Given a heap $\phi \in \mathrm{Hp}(n)$ and
$h \colon \ab \phi \to X$, we define a judgement $\J = \J(\phi, h)$ as
in~\eqref{eq:jweaken} by taking each $T_i$ to be $h(i)(x_{j_1}, \dots,
x_{j_k})$, where here $k = \mathsf{dp}_\phi(i)$ and $j_\ell =
\phi^{k-\ell}(i)$ for $\ell \in [k-1]$. This $\J$ will then satisfy
$(\phi_\J, h_\J) = (\phi, h)$ so long as it is in fact derivable in
the free $\{w\}$-\textsc{gat} on $X$.

We prove this by induction on $n$.  The base case $n = 1$ is clear;
suppose then that $n > 1$. If $\phi(i) = i-1$ for all $i > 0$, then
$\J(\phi, h)$ is a basic judgement, thus clearly
derivable. Otherwise, there must exist some $m < n$ which is not in
the image of $\phi$ (in the corresponding forest, such an $m$ amounts
to a leaf which is not the maximal node). Let $\res \phi m =
\res{\phi}{[0,m]}$ in $\mathrm{Hp}(m)$, and let $\phi \!\setminus\! m
\in \mathrm{Hp}(n-1)$ be defined by
\begin{equation}\label{eq:phiminus}
(\phi \!\setminus\! m) (i) = \begin{cases}
\phi(i) & i < m \rlap{ ;}\\
\phi(i+1) & i \geqslant m \text{ and } \phi(i+1) < m \rlap{ ;}\\
\phi(i+1) -1 & i \geqslant m \text{ and } \phi(i+1) > m\rlap{ ;}
\end{cases}
\end{equation}
(which amounts to stripping the leaf $m$ from the corresponding forest
and renumbering appropriately). Let $\res h m$ be the restriction of
$h$ to $[\res \phi m]$, and let $h\!\setminus\!m \colon \ab{\phi
  \!\setminus\! m} \to X$ take $i$ to $h(i)$ if $i < m$ and to
$h(i+1)$ if $i \geqslant m$. By induction, $\J(\res \phi m, \res h m)$
and $\J(\phi \!\setminus\! m, h \!\setminus\! m)$ are derivable
judgements, and it is easy to see that, up to $\alpha$-equivalence,
weakening $\J(\phi \!\setminus\! m, h \!\setminus\! m)$ with respect
to $\J(\res \phi m, \res h m)$ yields $\J(\phi, h)$. Thus this latter
judgement is derivable, and so the 
map~\eqref{eq:mappingweak} is indeed surjective.

This completes the proof on type judgements.  It remains to consider
term judgements. Any term judgement of the free $\{w\}$-\textsc{gat}
on $X$ is $\alpha$-equivalent to a unique one of the form
\begin{equation}\label{eq:jweaken2}
x_1 \ti T_1,\, \dots,\, x_{n-1} \ti T_{n-1} \vdash t : T_n\ty\rlap{ .}
\end{equation}
whose boundary is necessarily a type judgement $\J(\phi, h)$ of the
kind just described.  By induction on derivations, we may show that if
$T_n$ is $h(n)(x_{j_1}, \dots, x_{j_{k-1}})$ in~\eqref{eq:jweaken2},
then $t$ is $a(x_{j_1}, \dots, x_{j_{k-1}})$ for some $a \in \Tm X
{h(n)}$. We may thus identify $\Tm {WX} {\phi, h}$ with $\Tm X
{h(n)}$, as required.
\end{proof}

\begin{Rk}\label{rk:game2}
  The above description of the endofunctor $W$ reveals it to be
  identical in form to one arising in the analysis of innocent game
  semantics given in~\cite{Harmer2007Categorical}. Definition~10 of
  \emph{ibid}.~describes an exponential endofunctor $!$ (in fact a
  comonad) on a category $\cat{Gam}$ of games, whose meaning involves
  the addition of backtracking: thus, ``weakening = backtracking''.
  There are three main differences between our setting and that
  of~\cite{Harmer2007Categorical}, all relating to the category on
  which the endofunctor subsists.  Modulo these three differences, the
  endofunctors are completely identical in form.
  \begin{itemize}
  \item There are fewer objects in $\cat{Gam}$ than in $[\H^\op,
    \cat{Set}]$: games correspond to type-and-term structures without
    terms.\vskip0.5\baselineskip
  \item There are more morphisms in the category of games; such morphisms
    correspond to certain kinds of (decorated) relations, rather than to
    functions.\vskip0.5\baselineskip
  \item The category of games is \emph{polarised}, in that elements at
    even and odd degrees act with opposite variances, and the
    endofunctor $!$ adds backtracking only at odd degrees. One way of
    understanding this polarisation is to observe that the ``functional relations'' $X
    \to Y$ in $\cat{Gam}$ are not maps of $[\H^\op, \cat{Set}]$;
    rather, they correspond to diagrams
\begin{gather*}
 \vdots\\
\cd{
X(4) \ar[d]_\partial & \pushoutcorner[dl] \ar[l]^{} \bullet \ar[d]
\ar[r]_{f_4} & Y(4) \ar[d]_\partial \\
X(3) \ar[d]_\partial & \pushoutcorner \ar[l]^{f_3} \bullet \ar[d] \ar[r] & Y(3) \ar[d]_\partial \\
X(2) \ar[d]_\partial & \pushoutcorner[dl] \ar[l]^{} \bullet \ar[d] \ar[r]_{f_2} & Y(2) \ar[d]_\partial \\
X(1) & \ar[l]^{f_1} \bullet \ar@{=}[r] & Y(1)
}
\end{gather*}
with all marked squares pullbacks. The underlying cause of these
differences is as follows.  In type theory, a
context $(x\ti A\c y\ti B(x)\c z\ti C(x,y)\c w \ti D(x,y,z))$ is
thought of as $\Sigma x\ti A.\, \Sigma y\ti B(x).\, \Sigma z\ti
C(x,y).\, D(x,y,z)$. In game semantics, it would be interpreted as
$\Pi x\ti A.\, \Sigma y\ti B(x).\, \Pi z\ti C(x,y).\,
D(x,y,z)$.  \end{itemize}
\end{Rk}

\subsection{Unit and multiplication}
We now describe the  unit
 $\eta^W \colon 1 \Rightarrow W$ and the multiplication $\mu^W \colon WW
 \Rightarrow W$ of the weakening monad. The unit is quite straightforward.
 \begin{Defn}
   For any $n \geqslant 1$, we define $\gamma_n \in \mathrm{Hp}(n)$ by
   $\gamma_n(i) = 0$ if $i = 0$ and $\gamma_n(i) = i-1$ otherwise. For
   any $X \in [\H^\op, \cat{Set}]$ and $A \in X(n)$, we define $\tilde
   A \colon \ab{\gamma_n} \to X$ by $\tilde A(i) = \partial^{n-i}(A)$.
\end{Defn}
Note that $\hp_{\gamma_n}$ is just the usual linear ordering on $\{1,
\dots, n\}$. Moreover, $\ab{\gamma_n}$ is isomorphic to the
representable $\H(\thg, n)$, so that $\tilde A$ is simply the map
corresponding to $A \in X(n)$ under the Yoneda lemma.
 \begin{Prop}\label{prop:weakunit}
   For each $X \in [\H^\op, \cat{Set}]$, the unit map $\eta^W_X \colon X
   \to WX$ of the weakening monad has
   components
\begin{align*}
X(n) & \to WX(n) & \Tm X A & \to \Tm {WX} {\gamma_n, \tilde A}\\
A & \mapsto (\gamma_n, \tilde A) & a & \mapsto a\rlap{ .}
\end{align*}
\end{Prop}
\begin{proof}
The unit of the adjunction $G \dashv V$ at $X$ sends a type-element $A
\in X(n)$ to the equivalence class of the basic judgement 
\[x_1 \ti A_1,\, x_2 \ti A_2(x_1),\, \dots,\, x_{n-1} \ti
    A_{n-1}(x_1, \dots, x_{n-2})
    \vdash  A(x_1, \dots, x_{n-1}) \ty\rlap{ .}\]
Direct examination of the proof of Proposition~\ref{prop:wkendo} shows
that this element is $(\gamma_n, \tilde A) \in WX(n)$. We argue
similarly for term-elements.
\end{proof}

We now describe the multiplication of $W$, which will be slightly more
involved. First note that an element of $W^2X(n)$ is a pair $(\psi \in
\mathrm{Hp}(n), (\phi, h) \colon \ab \psi \to WX)$, where the second
component picks out pairs $(\phi_i \in
\mathrm{Hp}(\mathsf{dp}_\psi(i)), h_i \colon \ab{\phi_i} \to X)$ for
each $i \in [n]$ such that $(\partial \phi_i, \partial h_i) = (\phi_j,
h_j)$ whenever $\psi(i) = j$. For such a type-element we have $\Tm
{W^2X}{\psi, (\phi, h)} = \Tm {WX} {\phi_n, h_n} = \Tm {X}
{h_n(\mathsf{dp}_\psi(n))}$.

\begin{Defn}\label{def:weakoper}
  Given $(\psi, (\phi, h)) \in W^2X(n)$, we define the heap $\psi
  \star \phi \in \mathrm{Hp}(n)$ in relation form by
\begin{equation}\label{eq:relationstar}
i \hp_{\psi \star \phi} j \qquad \text{iff} \qquad i \hp_{\psi} j
\text{ and } \# i \hp_{\phi_j} \# j
\end{equation}
or in function form by $ (\psi \star \phi)(i) = \psi^{\# i- \phi_i(\#
  i)}(i)$; here, and elsewhere, we write $\# i$ as an abbreviation for
$\mathsf{dp}_\psi(i)$.  We define $\psi \star h \colon \ab{\psi \star
  \phi} \to X$ by $(\psi \star h)(i) = h_i(\# i)$. This is
well-defined by the observation that if $(\psi \star \phi)(i) = j$,
then $\# j = \phi_i(\#i)$, whence $\partial h_i(\# i) = h_i(\phi_i(\#
i)) = h_i(\# j) = h_j(\# j)$ (where the last equality holds since
$\partial h_i = h_j$) as required.
\end{Defn}


\begin{Ex}
Suppose that $\psi \in \mathrm{Hp}(6)$ and $\phi \colon \ab \psi \to W1$ are given by:
\[
\psi = \left(\cd[@R-1em@C-0.7em]{
& 6 \ar[d] \\
4 \ar[d] & 3 \ar[dr] & & 5 \ar[dl]\\
1 & & 2}\right) \quad
\phi  = \left( \cd[@R-1em@C-2.2em]{
& (3 \to 1 \leftarrow 2) \ar[d]_\partial \\
(1 \leftarrow  2) \ar[d]_\partial & (1 \leftarrow 2) \ar[dr]_\partial
& & (1\phantom{\leftarrow}\  2) \ar[dl]^\partial\\
(1) & & (1)}\right)
\]
then $\psi \star \phi$ is given by
\[
\psi \star \phi = \left(\cd[@R-1em@C-0.7em]{
& 6 \ar[ddr] \\
4 \ar[d] & 3 \ar[dr] & & 5 \\
1 & & 2}\right) \rlap{ .}
\]
\end{Ex}


 \begin{Prop}\label{prop:weakmult}
   For each $X \in [\H^\op, \cat{Set}]$, the multiplication $\mu^W_X
   \colon W^2X \to WX$ of the weakening monad has components
\begin{align*}
W^2X(n) & \to WX(n) & \Tm {W^2 X} {\psi, (\phi, h)} & \to \Tm {WX}
{\psi \star \phi, \psi \star h}\\
(\psi, (\phi, h)) & \mapsto (\psi \star \phi, \psi \star h) & a
 & \mapsto a\rlap{ .}
\end{align*}
\end{Prop}
\begin{proof}
  Let $(\psi, \ell) = (\psi, (\phi, h)) \in W^2X(n)$. Note that if $\psi
  = \gamma_n$, then we have $\mu^W_X(\gamma_n, (\phi, h)) =
  \mu^W_X(\eta^W_{WX}(\phi_n, h_n)) = (\phi_n, h_n)$, and also by direct
  calculation that $(\gamma_n \star \phi, \gamma_n \star h) = (\phi_n,
  h_n)$. So without loss of generality we
  may assume $\psi \neq \gamma_n$. We proceed by induction on $n$. The
  base case $n = 1$ is clear, since the only $\psi \in \mathrm{Hp}(1)$
  is $\gamma_1$. So assume $n>1$ and $\psi \neq \gamma_n$. Then there
  exists some $m < n$ not in the image of $\psi$, and as in the proof
  of Proposition~\ref{prop:wkendo}, we form $(\res \psi m, \res \ell m)
  \in W^2X(m)$ and $(\psi \!\setminus\! m, \ell \!\setminus\! m) \in
  W^2X(n-1)$. By induction and a direct
  calculation, we have that
  \begin{gather*}
    \mu^W_X(\res \psi m, \res \ell m) = (\res \psi m
\star \res \phi m,\, \res \psi m \star \res h m)
=(\res{(\psi \star \phi)}{m}, \res{(\psi \star
  h)}{m})\\
\text{and }
\mu^W_X(\psi \!\setminus\! m, \ell
\!\setminus\! m) = (\psi
\!\setminus\! m \star \phi \!\setminus\! m,\, \psi \!\setminus\! m 
\star h \!\setminus\! m) = ((\psi \star \phi) \!\setminus\! m, (\psi \star h)
\!\setminus\! m)\rlap{ .}
\end{gather*}

Now the judgement $\J(\psi, \ell)$ is derivable by weakening $\J(\psi
\!\setminus\!  m, \ell \!\setminus\! m)$ with respect to $\J(\res \psi
m, \res \ell m)$ and $\alpha$-converting. Since $\mu^W_X$ is the image
under the forgetful functor $V$ of the interpretation $\epsilon_{GX}
\colon GVGX \to GX$, and interpretations preserve derivations, it
follows that the judgement represented by $\mu^W_X(\psi, \ell)$ may be
derived by weakening $\J((\psi \star \phi) \!\setminus\! m, (\psi
\star h) \!\setminus\! m)$ with respect to $\J(\res{(\psi \star
  \phi)}{m}, \res{(\psi \star h)}{m})$ and $\alpha$-converting. But
the judgement so obtained is easily seen to be $\J(\psi \star \phi,
\psi \star h)$, so that finally $\mu^W_X(\psi, \ell) = (\psi \star
\phi, \psi \star h)$ as required. 

 This completes the argument for
type-elements. That  for terms is similar; the key point is that if
$a \in \Tm {W^2X}{\psi, \ell}$ with $\psi \neq
\gamma_n$, then on taking $m < n$ with $m \notin
\im \psi$ and forming $(\res \psi m, \res \ell m) \in
W^2X(m)$ and $(\psi \! \setminus \!  m, \ell \! \setminus\!  m) \in
W^2X(n-1)$, we now have $a \in \Tm{W^2X}{\psi \! \setminus \!  m, \ell \!
  \setminus\!  m}$. Weakening $\J(\psi
\! \setminus \!  m, \ell \!  \setminus\!  m, a)$ with respect to 
$\J(\res \psi m, \res \ell m)$ yields
back $\J(\psi, \ell, a)$, and we conclude the argument as before
using induction and the preservation of derivations by $\mu^W_X$.
\end{proof}

\begin{Rk}\label{rk:game3}
  Again, we may link our monad $W$ to the exponential comonad $!$
  of~\cite{Harmer2007Categorical}. The obvious difference is that one
  is a monad and the other a comonad. But this is easily accounted for
  due to the polarisation present in the category of games: the
  comonad $!$ corresponds to adding backtracking at ``contravariant''
  odd degrees (there is a corresponding monad $?$ which adds
  backtracking at even degrees). The issue of polarity aside, the
  comonad structure of $!$ corresponds exactly to the monad structure
  of $W$ described above.
\end{Rk}

\section{Weakening and projection}\label{sec:project}
We now turn to the additional structure imposed on a type-and-term
structure by the projection rule. As we have already explained, the
projection rule is not well-behaved in the absence of the weakening
rule. Consequently, in this section, we will seek to characterise the
\emph{weakening-and-projection monad} $P$ induced by the
free-forgetful adjunction $\{w,p\}\text-\cat{GAT} \leftrightarrows
[\H^\op, \cat{Set}]$.

\subsection{Underlying endofunctor}
As before, we begin by describing the underlying endofunctor of the
weakening-and-projection monad. 
\newcommand{\wk}{\mathsf{wk}}

\begin{Prop}\label{prop:projendo}
  The value at $X \in [\H^\op, \cat{Set}]$ of the underlying
  endofunctor $P$ of the weakening-and-projection monad agrees with
  $W$ on type-elements, and on term-elements is given by
\begin{align*}
  \Tm{PX}{\phi, h} = \Tm {X} {h(n)} + \{ \pi_i : i \in
  [n-1]\text{, }
  \phi(n) = \phi(i)\text{, } h(n) = h(i)\}\rlap{ .}
\end{align*}
\end{Prop}
Note that the left-hand summand above is $\Tm{WX}{\phi, h}$, so that
$PX$ is simply the extension of $WX$ by the addition of new
term-elements representing projections. We write
$\theta_X \colon WX \to PX$ for the evident inclusion maps.
\begin{proof}
  By Proposition~\ref{prop:commonanalysis}, elements
  of $PX$ are $\alpha$-equivalence classes of type or term judgements
  of the free $\{w,p\}$-\textsc{gat} on $X$. The type judgements are
  visibly the same as those of the free $\{w\}$-\textsc{gat}, while
  the term judgements augment those of the free $\{w\}$-\textsc{gat}
  with ones of the form\begin{equation}\label{eq:jproject}
    \J(\phi, h, \pi_i) = x_1 \ti T_1,\, \dots,\, x_{n-1} \ti T_{n-1}
    \vdash x_i : T_{i}
\end{equation}
for $i \in [n-1]$ such that $T_n = T_i$; i.e., such that $\phi(n)
= \phi(i)$ and $h(n) = h(i)$. This accounts for the right-hand summand
in $\Tm {PX}{\phi, h}$. 
\end{proof}

\subsection{Unit and multiplication}
We now describe the unit $\eta^P \colon 1 \Rightarrow P$ and the
multiplication $\mu^P \colon PP \Rightarrow P$ of the
weakening-and-projection monad. 

\begin{Prop}\label{prop:projunit}
  For each $X \in [\H^\op, \cat{Set}]$, the unit $\eta^P_X \colon X
  \to PX$ of the weakening-and-projection monad is the composite
  $\theta_X \circ \eta^W_X \colon X \to WX \to PX$.
\end{Prop}
 \begin{proof}
  Clear.
\end{proof}
We now turn to the multiplication, for which we will need to identify
term-elements of $P^2X$ over a type-element $(\psi, (\phi, h)) \in
P^2X(n) = W^2X(n)$. By definition $\Tm{P^2X}{\psi, (\phi, n)}$ is the
set
\begin{align*}
  \Tm{PX}{\phi_n, h_n} + \{\pi_i(\psi, (\phi, h)) \mid i \in [n-1],
  \psi(n) = \psi(i), \phi_n = \phi_i, h_n = h_i\}
\end{align*}
where we annotate the projections in the second factor to distinguish
them from those appearing in the further decomposition of
$\Tm{PX}{\phi_n, h_n}$ as
\begin{align*}
\Tm{X}{h_n(\# n)} + \{\pi_i(\phi_n, h_n) \mid i \in [\# n -1], \phi_n(\# n)
= \phi_n(i), h_n(\# n) = h_n(i)\}\rlap{ .}
\end{align*}

\begin{Prop}\label{prop:projmult}
  For each $X \in [\H^\op, \cat{Set}]$, the multiplication $\mu^P_X
  \colon P^2X \to PX$ of the weakening-and-projection monad agrees
  with that of $W$ on type-elements, and on term-elements is defined
  at $(\psi, (\phi, n)) \in P^2X(n)$ to be the mapping $\Tm {P^2X}
  {\psi, (\phi, h)} \to \Tm {PX} {\psi \star \phi, \psi \star h}$
  given by
\begin{gather*}
a \mapsto  \begin{cases}
a & \text{if $a \in \Tm X {h_n(\# n)}$;}\\
\pi_i & \text{if $a = \pi_i(\psi, (\phi, h))$;}\\
\pi_{\psi^{\# n - i}(n)} & \text{if $a =
  \pi_i(\phi_n, h_n)$.}
\end{cases}
\end{gather*}
\end{Prop}
\begin{proof}
  The assertion concerning type-elements is clear. As for
  term-elements, an element of $\Tm {P^2X}{\psi, (\phi, h)}$ that lies
  in $\Tm{X}{h_n(\# n)} = \Tm {W^2X}{\psi, (\phi, h)}$ represents a
  term judgement derived without the use of projection, so that the
  action of the multiplication is inherited from $W$.  For an element
  of the form $\pi_i(\psi, (\phi, h))$, the judgement of the free
  $\{w,p\}$-\textsc{gat} on $WX$ which it represents is derivable from
  the judgement representing $(\psi, (\partial \phi, \partial h))$ by
  a single application of the projection rule. As in the proof of
  Proposition~\ref{prop:weakmult}, it follows that the judgement
  represented by $\mu^P_X(\pi_i(\psi, (\phi, h)))$ is derivable from
  that representing $\mu^P_X(\psi, (\partial \phi, \partial h)) =
  (\partial(\psi \star \phi), \partial(\psi \star h))$ by applying the
  same instance of the projection rule; it follows that
  $\mu^P_X(\pi_i(\psi, (\phi, h))) = \pi_i$ as required.

  Finally consider an element $\pi_i(\phi_n, h_n) \in \Tm {P^2X}{\psi,
    (\phi, h)}$.  If $\psi = \gamma_n$, then $(\psi, (\phi, h)) =
  \eta^P_{PX}(\phi_n, h_n)$ and the given term-element is the image
  under $\eta^P_{PX}$ of $\pi_i \in \Tm {PX} {\phi_n, h_n}$. It
  follows that applying $\mu^P_X$ yields back $\pi_i =
  \pi_{\gamma_n^{\# n - i}(n)}$, as required.  For the case $\psi \neq
  \gamma_n$ we now proceed by induction on $n$; in what follows we
  abbreviate $\ell = (\phi, h)$. The base case $n = 1$ is trivial as
  then necessarily $\psi = \gamma_1$. So assume $n>1$ and $\psi \neq
  \gamma_n$. Then there is some $m < n$ not in the image of $\psi$,
  and as in the proofs of Propositions~\ref{prop:wkendo}
  and~\ref{prop:weakmult} we may form the type-elements $(\res \psi m,
  \res \ell m)$ and $(\psi \!\setminus\! m, \ell
  \!\setminus\! m)$ and the term-element $\pi_i(\phi_n,
  h_n) \in \Tm{P^2X}{\psi \!\setminus\! m, \ell \!\setminus\! m}$. Now
  in the free $\{w,p\}$-\textsc{gat} on $PX$, weakening the judgement
  represented by this term-element with respect to $\J(\res \psi m,
  \res \ell m)$ yields the judgement representing 
$\pi_i(\phi_n, h_n) \in \Tm {P^2X}{\psi, \ell}$. As in the proof of
Proposition~\ref{prop:weakmult}, applying $\mu^P_X$ to $(\res \psi m,
\res \ell m)$ yields $(\res{(\psi \star \phi)}{m}, \res{(\psi \star
  h)}{m})$, while by the inductive hypothesis, applying it 
to $\pi_i(\phi_n, h_n) \in \Tm{P^2X}{\psi \!\setminus\! m,
    \ell \!\setminus\! m}$ yields
$
\pi_j \in \Tm {PX}{(\psi \star \phi) \!
\setminus \! m, (\psi \star h) \!
\setminus \! m}
$
where here \[j = {(\psi \setminus m)^{\mathsf{dp}_{\psi \setminus m}(n - 1)
    - i}(n-1)} = (\psi \setminus m)^{\# n - i}(n-1)\rlap.\] Since $m$ is not
in the image of $\psi$, it follows easily from~\eqref{eq:phiminus}
that $j = \psi^{\# n - i}(n)$ if $i < m$ and $j = \psi^{\# n - i}(n) -
1$ if $i > m$. Weakening the judgement represented by this
term-element with respect to $\J(\res{(\psi \star \phi)}{m},
\res{(\psi \star h)}{m})$ is easily seen (in either of the two cases
$i < m$ and $i > m$) to yield the judgement represented by
$\pi_{\psi^{\# n - i}(n)}$, as
required. We conclude as before by using the fact that $\mu^P_X$
preserves derivations.
\end{proof}

\section{Substitution}
\label{sec:subst}
Our next step will be to consider the  structure imposed on a type-and-term
structure by the substitution rules, thus describing the 
\emph{substitution monad} $S$ induced by the free-forgetful adjunction $
\{s\}\text-\cat{GAT} \leftrightarrows [\H^\op, \cat{Set}]$.

\subsection{Underlying endofunctor}
We begin by  describing the underlying endofunctor of the
substitution monad. Whereas the combinatorics of weakening are
controlled by min-heaps, those of substitution are controlled by
increasing lists of natural numbers.
\begin{Defn}
  A \emph{inc-list} of length $n$ is a function $\alpha \colon [0, n]
  \to \mathbb N$ such that $\alpha(0) = 0$ and $\alpha(i) > \alpha(j)$
  whenever $i > j$.  We write $\mathrm{Inc}(n)$ for the set of
  inc-lists of length $n$. Given $\alpha \in \mathrm{Inc}(n+1)$, we
  write $\partial(\alpha)$ for $\res \alpha {[0,n]} \in
  \mathrm{Inc}(n)$.
\end{Defn}

An inc-list $\alpha$ can be seen as encoding the shape of a type
judgement in a free $\{s\}$-\textsc{gat}. The length of the list gives
the degree of the judgement, while the values $\alpha(1), \dots,
\alpha(n)$ indicate the degrees of the individual types appearing in
it. The case where $\alpha(i) = i$ for each $i$ encodes a judgement
without substitution; otherwise, we have values $\alpha(m)$ and
$\alpha(m+1)$ that are not consecutive, and this must be compensated
for by the substitution of suitable terms into the $m+1$st type to
reduce its degree to merely one greater than that of the $m$th type. The
following definition make this precise.

\begin{Defn}\label{def:inclist}
Given $\alpha \in \mathrm{Inc}(n)$, we define the presheaf $[\alpha] \in
  [\H^\op, \cat{Set}]$ encoding a type judgement of shape $\alpha$ by:
\begin{align*}
\ab\alpha(i) &= \begin{cases}
\{i\} & \text{if $i \leqslant \alpha(n)$;} \\
\emptyset & \text{otherwise;}
\end{cases}
& \text{and }\ 
\ab\alpha(i_t) &= \begin{cases}
\{i_t\} & \text{if $i \leqslant \alpha(n)$, $i \notin \im \alpha$; }\\
\emptyset & \text{otherwise.}
\end{cases}
\end{align*}
with the unique possible boundary maps.
\end{Defn}
For example, if $\alpha \in \mathrm{Inc}(3)$ has values $0 < 2 < 3
< 5 < 8$, then $\ab \alpha$ is given by:
\[
\cd[@C-1.2em]{
& \{1_t\} &  \emptyset & \emptyset & \{4_t\} & \emptyset & \{6_t\} &
\{7_t\} & \emptyset\\
\{1\} \ar@{<-}[ur] \ar@{<-}[r] & \{2\} \ar@{<-}[r] \ar@{<-}[ur]
& \{3\} \ar@{<-}[r] \ar@{<-}[ur] & \{4\} \ar@{<-}[r] \ar@{<-}[ur] &
\{5\} \ar@{<-}[ur] \ar@{<-}[r]
& \{6\} \ar@{<-}[r] \ar@{<-}[ur] & \{7\} \ar@{<-}[r] \ar@{<-}[ur] &
\{8\} \ar@{<-}[ur] }
\]


For a general $\alpha \in \mathrm{Inc}(n)$, a map $h \colon \ab \alpha
\to X$ in $[\H^\op, \cat{Set}]$ is determined by giving, firstly, a
type-element $h(\alpha(n)) \in X(\alpha(n))$---which determines $h(i)
\in X(i)$ for each smaller $i$ by taking iterated boundaries---and
secondly, term-elements $h(i_t) \in \Tm {X} {h(i)}$ for each $i \in
[\alpha(n)] \setminus \im \alpha$.  When $n > 1$, we write $\partial
h \colon \ab{\partial \alpha} \to X$ for the restriction of $h$ along the obvious inclusion
$\ab{\partial \alpha} \to \ab \alpha$.

\begin{Prop}\label{prop:substendo}
  The value at $X \in [\H^\op, \cat{Set}]$ of the underlying
  endofunctor of the substitution monad has type-elements and
  boundaries given by
   \begin{gather*}
  SX(n) = \sum_{\phi \in \mathrm{Inc}(n)}\, [\H^\op, \cat{Set}](\ab
  \alpha, X)  \qquad
\begin{aligned}[t]
  \partial \colon
SX(n+1) & \to SX(n) \\ 
 (\alpha, h) & \mapsto (\partial \alpha, \partial h)
\end{aligned}
\end{gather*}
and term-elements $\Tm {SX} {\alpha, h} = \Tm X {h(\alpha(n))}$ for each $(\alpha,
h) \in SX(n)$.
\end{Prop}
\begin{proof}
  We prove by induction on derivations that, if $\J$ is a representative
  type judgement~\eqref{eq:jweaken} of the free
  $\{s\}$-\textsc{gat} on $X$, then there are natural numbers $0 <
  \alpha(1) < \dots < \alpha(n)$, type-elements $A_i \in X(\alpha(i))$
  for $i \in [n]$, and term-elements $a_i \in X(i_t)$ for $i \in
  [\alpha(n)] \setminus
  \im \alpha$, such that each $T_i$ in $\J$ is of the form
\begin{equation}\label{eq:subst1}
A_i(x_1, \dots, x_{\alpha(i)-1})[t_{\alpha(i) - 1} / x_{\alpha(i) - 1}]\dots[t_1 / x_1][x_1 /
x_{\alpha(1)}]\dots [x_{i-1} / x_{\alpha(i-1)}]\rlap{ ;}
\end{equation}
here, $t_i$ is the expression $x_i$ if $i \in \im \alpha$ and is
$a_i(x_1, \dots, x_{i-1})$ otherwise.
By a further induction on derivations, we may show that each $A_i$ and
each $\partial(a_i)$ is of the form $\partial^\ell(A_n)$ for a
suitable $\ell$.  It follows that we have a well-defined map $h \colon
\ab \alpha \to X$ given by $h(i) = \partial^{\alpha(n) - i}(A_n)$ and
$h(i_t) = a_i$; and that the pair $(\alpha, h)$ encodes all the
information of the type judgement $\J$.  

A similar induction on derivations shows that  a
 term judgement $\J'$ of the form~\eqref{eq:jweaken2} in the free
$\{s\}$-\textsc{gat} on $X$ is given by a type judgement as above
together with a term expression $t$ of the form
\begin{equation*}
a(x_1, \dots, x_{\alpha(n)-1})[t_{\alpha(n) - 1} / x_{\alpha(n) - 1}]\dots[t_1 / x_1][x_1 /
x_{\alpha(1)}]\dots [x_{i-1} / x_{\alpha(n-1)}]
\end{equation*}
for some $a \in \Tm X {A_n} = \Tm X {h(\alpha(n))}$. We thus have maps
\begin{equation}
\label{eq:mappingsubst}
\begin{aligned}
SX(n) & \to\!\!  \sum_{\alpha \in \mathrm{Inc}(n)} \![\H^\op,
\cat{Set}](\ab \alpha, X) &
\Tm {SX}{[\J]} & \to \Tm {X} {h(\alpha(n))}
\\
[\J] & \mapsto (\alpha_\J, h_\J) &
[\J'] & \mapsto a_{\J'}
\end{aligned}
\end{equation}
which by the above are well-defined, injective and compatible with
boundaries. It remains to prove their surjectivity.  Given $\alpha \in
\mathrm{Inc}(n)$ and $h \colon \ab \alpha \to X$, by taking $a_i =
h(i_t)$ (for $i \in [\alpha(n)] \setminus \im \alpha$) and defining
$T_i$ as in~\eqref{eq:subst1}, we obtain a type judgement $ \J(\alpha,
h)$ of the form~\eqref{eq:jweaken}, whose image
under~\eqref{eq:mappingsubst} will be $(\alpha, h)$ so long as it is
in fact a derivable type judgement of the free $\{s\}$-\textsc{gat} on
$X$. Similarly, to any $a \in \Tm {X} {h(\alpha(n))}$ we may assign
a term judgement $\J(\alpha, h, a)$ with boundary $\J(\alpha, h)$
which will be sent to $a$ by the right-hand map in~\eqref{eq:mappingsubst}
so long as it is in fact derivable.

We prove derivability of these two kinds of judgements simultaneously
by induction on the value $\alpha(n) - n$. In the base case $\alpha(n)
= n$ we see easily that $\J(\alpha,h)$ and $\J(\alpha, h,a)$ are basic
judgements of the free $\{s\}$-\textsc{gat} on $X$, and so derivable. For the
inductive step, suppose that $\alpha(n) - n > 0$, and we wish to
derive $\J(\alpha, h)$. Choose some $m < n$ and $j \in \mathbb N$ such
that $\alpha(m) < j < \alpha(m+1)$. We now define $\alpha_j \in
\mathrm{Inc}(n+1)$ and $ \alpha^ j \in \mathrm{Inc}(m+1)$ by
\[
(\alpha_j)(i) = \begin{cases}
\alpha(i) & i \leqslant m \rlap{ ;} \\
j & i = m +1 \rlap{ ;} \\
\alpha(i-1) & i > m+1 \rlap{ ;}
\end{cases} \qquad \text{and} \qquad
\alpha^j = \res{(\alpha_j)}{[0, m+1]}\rlap{ .}
\]

Note that $\alpha_j(n+1) = \alpha(n)$ since $n + 1 > m + 1$, and so
$\alpha_j(n+1) - (n+1) < \alpha(n) - n$.  Likewise $\alpha^j(m+1) = j <
\alpha(m+1)$ and so $\alpha^j(m+1) - (m+1) < \alpha(m+1) - (m+1)
\leqslant \alpha(n) - n$. Note further that $\ab{\alpha_j}$ and
$[\alpha^j]$ are subpresheaves of $\ab\alpha$; we write $h_j$ and
$h^j$ for the restrictions of $h$ to them.  By induction,
$\J(\alpha_j,\, h_j)$ is a derivable type judgement and
$\J(\alpha^j,\, h^j, h(j_t))$ a derivable term judgement, and
substituting the latter into the former (and $\alpha$-converting) now
yields the required derivation of $\J(\alpha, h)$.  In a similar
manner, each term judgement $\J(\alpha, h,a)$ may be derived
inductively from $\J(\alpha_j, h_j, a)$ and $\J(\alpha^j, h^j, h(j_t))$.
\end{proof}

\subsection{Unit and multiplication}
We now describe the  unit
 $\eta^S \colon 1 \Rightarrow S$ and the multiplication $\mu^S \colon SS
 \Rightarrow S$ of the substitution monad.

 \begin{Defn}
   For any $n \geqslant 1$, we write $\iota_n \in \mathrm{Inc}(n)$ for
   the inc-list given by $\iota_n(i) = i$. For any $X \in [\H^\op,
   \cat{Set}]$ and $A \in X(n)$, we define $\tilde A \colon \ab{
   \iota_n} \to X$ by $\tilde A(i) = \partial^{n-i}(A)$ for each $i \in
   [n]$.
 \end{Defn}
 As before, $\ab{\iota_n}$ is isomorphic to the representable
 $\H(\thg, n)$ so that $\tilde A$ corresponds to $A \in X(n)$ under
 the Yoneda lemma.

 \begin{Prop}\label{prop:substunit}
   For each $X \in [\H^\op, \cat{Set}]$, the unit map $\eta^S_X \colon X
   \to SX$ of the substitution monad has
   components
\begin{align*}
X(n) & \to SX(n) & \Tm X A & \to \Tm {SX} {\iota_n, \tilde A}\\
A & \mapsto (\iota_n, \tilde A) & a & \mapsto a\rlap{ .}
\end{align*}
\end{Prop}
\begin{proof}
Immediate from examination of the proof of
Proposition~\ref{prop:substendo}.
\end{proof}
We now turn to the multiplication of $S$, for which we need an
explicit description of $S^2X$.
Given $\alpha \in \mathrm{Inc}(n)$, a map $[\alpha] \to SX$ is
determined as in the discussion following
Definition~\ref{def:inclist} by its value at $\alpha(n)$ and its
values at $i_t$ for each $i \in \ab{\alpha(n)} \setminus \im
\alpha$. Giving these data amounts to giving
\begin{itemize}
\item An element $(\beta \in \mathrm{Inc}(\alpha(n)),\, h \colon \ab \beta \to X) \in
SX(\alpha(n))$; and
\item Elements $k(i) \in \Tm {SX}{\partial^{\alpha(n) -
    i}(\beta, h)} = \Tm X {h(\beta(i))}$ for  $i \in
  [\alpha(n)] \setminus \im \alpha$.
\end{itemize}
Thus we write a typical type-element of $S^2X$ as $(\alpha \in
\mathrm{Inc}(n), (\beta, h, k) \colon \ab \alpha \to SX)$. Now on
terms, we have $\Tm {S^2X}{\alpha, (\beta, h, k)} = \Tm {SX}{\beta, h}
= \Tm X {h(\beta(\alpha(n)))}$.




\begin{Defn}
  Given $(\alpha, (\beta, h, k)) \in S^2X(n)$, we define the inc-list $\beta \alpha \in \mathrm{Inc}(n)$ by $(\beta\alpha)(i) =
  \beta(\alpha(i))$ and define $h \cup k \colon \ab{\beta
    \alpha} \to X$ by taking $(h \cup k)(i) = h(i)$ for $i \in
  [\beta\alpha(n)]$ and for $i \in [\beta\alpha(n)] \setminus \im
  \beta \alpha$ taking
\[
(h \cup k)(i_t) = \begin{cases}
h(i_t) & \text{for $i \notin \im \beta$;} \\
k(j) & \text{for $i = \beta(j)$, $j
  \notin \im \alpha$.}
\end{cases}
\]
\end{Defn}
 \begin{Prop}\label{prop:substmult}
   For each $X \in [\H^\op, \cat{Set}]$, the multiplication $\mu^S_X
   \colon S^2X \to SX$ of
   the substitution monad has components
\begin{align*}
S^2X(n) & \to SX(n) & \Tm {S^2 X} {\alpha, (\beta, h, k)} & \to \Tm {SX}
{\beta \alpha, h \cup k}\\
(\alpha, (\beta, h, k)) & \mapsto (\beta \alpha, h \cup k) & a
 & \mapsto a\rlap{ .}
\end{align*}
\end{Prop}
\begin{proof} We prove the result for type- and term-elements
  simultaneously by induction on $\alpha(n) - n$. Consider first a
  type-element $(\alpha, \ell) = (\alpha, (\beta, h, k)) \in
  S^2X(n)$. In the base case where $\alpha(n) = n$, we must have
  $\alpha = \iota_n$ and now $\mu^S_X(\iota_n, (\beta, h, k)) =
  \mu^S_X(\eta^S_{SX}(\beta, h)) = (\beta, h)$, which is visibly equal
  to $(\beta \iota_n, h \cup k)$ (since in this case $k$ is
  trivial). For the inductive step, assume $\alpha(n) > n$. As in the
  proof of Proposition~\ref{prop:substendo}, we can find some $m < n$
  and some $j \in \mathbb N$ with $\alpha(m) < j < \alpha(m+1)$ and
  now form the type-elements $(\alpha_j, \ell_j) \in S^2X(n+1)$ and
  $(\alpha^j, \ell^j) \in S^2X(m+1)$ and term-element $\ell(j_t) =
  k(j) \in \Tm {S^2X} {\alpha^j, \ell^j}$.

  Now $\ell_j = (\beta, h, k') \colon \ab{\alpha_j} \to X$, where $k'$
  is obtained from $k$ by removing the value at $j$; while $\ell^j =
  (\res \beta {[0,j]}, h', k'')$, where where $h'$ is the restriction
  of $h$ along the inclusion $[\res\beta{[0,j]}] \to [\beta]$ and
  $k''$ is the restriction of $k$ to $[j-1]$. Thus
  by induction and direct calculation, we see that
  \begin{gather*}
    \mu^S_X(\alpha_j, \ell_j) = 
    (\beta \alpha_j, h \cup k')
    = ((\beta \alpha)_{\beta(j)}, (h \cup k)_{\beta(j)})\\
\text{and }
    \mu^S_X(\alpha^j, \ell^j) = 
    ((\res \beta {[0,j]})\alpha^j, h' \cup k'')
    = ((\beta\alpha)^{\beta(j)}, (h \cup k)^{\beta(j)})\rlap{ ,}
  \end{gather*}
  and that $k(j) \in \Tm {S^2X}{\alpha^j, \ell^j}$ is sent to $k(j)
  \in \Tm {SX} {(\beta\alpha)^{\beta(j)}, (h \cup k)^{\beta(j)}}$.
  Now the judgement $\J(\alpha, \ell)$ of the free
  $\{s\}$-\textsc{gat} on $SX$ is derivable by substituting
  $\J(\alpha^j, \ell^j, k(j))$ into $\J(\alpha_j, \ell_j)$ and
  $\alpha$-converting, and it follows that the judgement represented by
  $\mu^S_X(\alpha, \ell)$ may be derived by substituting
  $\J((\beta\alpha)^{\beta(j)}, (h \cup k)^{\beta(j)}, k(j))$ into
  $\J((\beta\alpha)_{\beta(j)}, (h \cup k)_{\beta(j)})$ and
  $\alpha$-converting; whence $\mu^S_X(\alpha,
  (\beta, h, k)) = (\beta \alpha, h \cup k)$ as required. The argument
  for term judgements is similar, and hence omitted.
\end{proof}

\section{Combining the structures}\label{sec:combining}
We now combine the results of the preceding three sections in order to
describe the structure imposed on a type-and-term structure by all the
deduction rules of generalised algebraic theories; we will thus
describe the monad $T$ for \textsc{gat}s induced by the free-forgetful
adjunction $\cat{GAT} \leftrightarrows [\H^\op, \cat{Set}]$.

Let $P$ and $S$ denote, as in the preceding sections, the weakening-and-projection monad and
the substitution monad on $[\H^\op, \cat{Set}]$. We have natural
transformations $\rho \colon P \Rightarrow T \Leftarrow S \colon
\sigma$ expressing that every derivable
judgement of the free $\{w,p\}$- or $\{s\}$-\textsc{gat} on some $X$ is also
derivable in the free \textsc{gat} on $X$. It is easy to
see that $\rho$ and $\sigma$ are compatible with the unit and
multiplication maps and so exhibit $P$ and $S$ as submonads of
$T$. Our task in this section will be to describe how these submonads
combine together to yield $T$.

\subsection{Underlying endofunctor}
We first characterise the underlying endofunctor of the monad $T$ for
\textsc{gat}s in terms of those of the weakening-and-projection and
substitution monads. Our result expresses that every judgement of the
free \textsc{gat} on $X$ may be obtained in a unique way (up to
$\alpha$-conversion) by first applying substitution to basic
judgements, and then weakening and projection to these substituted
judgements.

\begin{Prop}
For any $X \in [\H^\op, \cat{Set}]$ the composite map
\begin{equation}\label{eq:compositeendo}
  \kappa_X \defeq PSX \xrightarrow{P\tau_X} PTX
  \xrightarrow{\sigma_{TX}} TTX \xrightarrow{\mu^T_X} TX
\end{equation}
is invertible.
\end{Prop}
Note first that a type-element of $PSX$
has the form $(\phi, (\alpha, h))$, where $\phi \in \mathrm{Hp}(n)$
and $(\alpha_i, h_i) \in SX(\# i)$ for each $i \in [n]$, such that
$(\partial \alpha_i, \partial h_i) = (\alpha_j, h_j)$ whenever
$\phi(i) = j$. Furthermore, $\Tm {PSX} {\phi, (\alpha, h)}$ is given
by the sum
\begin{equation}
  \label{eq:psxterm}
\Tm {SX}{\alpha_n, h_n} +
\{\pi_i \mid i \in [n-1], \phi(n) = \phi(i),
\alpha_n = \alpha_i, h_n = h_i\}\rlap{ .}
\end{equation}

\begin{proof}
  Consider first a representative type judgement $\J$ of the
  form~\eqref{eq:jweaken} in the free \textsc{gat} on $X$. By
  induction on derivations, we show that for each $i \in [n]$ there
  are $0 < j_1 < \dots < j_{k-1} < i$ such that $\mathsf{fv}(T_i) =
  \{x_{j_1}, \dots, x_{j_{k-1}}\}$ and such that
  \begin{equation}
    \label{eq:combine}
  x_{j_1} : T_{j_1}, \dots, x_{j_{k-1}} : T_{j_{k-1}} \vdash T_i\ty
\end{equation}
is derivable in the free $\{s\}$-\textsc{gat} on $X$. Define $\phi(i)$
to be $0$ if $k = 1$ and to be $j_{k-1}$ otherwise, and define
$(\alpha_i, h_i) \in SX(k)$ to be the element representing the
$\alpha$-equivalence class of~\eqref{eq:combine}. Taking also $\phi(0)
= 0$ we obtain a heap $\phi \in \mathrm{Hp}(n)$; moreover, for those
$i$ with $\phi(i) > 0$ we see by a further induction on derivations
that $\mathsf{fv}(T_{j_{k-1}}) = \{x_{j_1}, \dots, x_{j_{k-2}}\}$,
whence $\partial(\alpha_i, h_i) = (\alpha_{\phi(i)},
h_{\phi(i)})$. Thus we have a well-defined map $(\alpha, h) \colon \ab
\phi \to SX$ and so an element $(\phi, (\alpha, h)) \in PSX(n)$. In
this way, we have defined a mapping
\begin{align*}
\theta \colon TX(n) &\to PSX(n)\\
 [\J] &\mapsto (\phi_\J, (\alpha_\J, h_\J))
\end{align*}
which we claim is inverse to the $n$-component
of~\eqref{eq:compositeendo}.  It is easy to see that $\theta$ is
injective, so it is enough to show that $1 = \theta \circ \kappa_X \colon
PSX(n) \to TX(n) \to PSX(n)$.

So let $(\phi, \ell) = (\phi, (\alpha, h)) \in PSX(n)$. If $\phi =
\gamma_n$, then $(\phi, \ell) = \eta^P_{SX}(\alpha_n, h_n)$. Now by
direct calculation $\kappa \circ \eta^P S = \tau$ so that
$\kappa_X(\phi, \ell) = \tau_X(\alpha_n, h_n)$ represents the
judgement $\J(\alpha_n, h_n)$ of the free \textsc{gat} on $X$. But by
inspection, the image of $\J(\alpha_n, h_n)$ under $\theta$ is again
$(\gamma_n, (\alpha, h))$, as required.  For the case $\phi \neq
\gamma_n$ we proceed by induction on $n$. The case $n = 1$ is trivial;
so assume $n > 1$. As in the proof of Proposition~\ref{prop:wkendo},
we may find $m < n$ such that $m \notin \im \phi$ and form the
type-elements $(\res \phi m, \res \ell m)$ and $(\phi \!\setminus\! m,
\ell \!\setminus\!  m)$ of $PSX$. Now the judgement $\J(\phi, \ell)$
of the free $\{w,p\}$-\textsc{gat} on $SX$ is derivable by weakening
$\J(\phi \!\setminus\!  m, \ell \!\setminus\! m)$ with respect to
$\J(\res \phi m, \res \ell m)$ and $\alpha$-converting. It follows
that the judgement represented by $\kappa_X(\phi, \ell)$ is obtained
in the same way from the judgements represented by $\kappa_X(\res \phi
m, \res \ell m)$ and $\kappa_X(\phi \!\setminus\! m, \ell
\!\setminus\!  m)$. But by induction, these latter judgements are sent
to $(\res \phi m, \res \ell m)$ and $(\phi \!\setminus\! m, \ell
\!\setminus\! m)$ by $\theta$, and now by direct inspection of the
description of $\theta$ given above, we conclude that
$\theta(\kappa_X(\phi, \ell)) = (\phi, \ell)$, as required.

This completes the argument on type judgements; that on
term judgements is similar. The key point is that we may show by
induction on derivations that a typical term judgement $\J'$ of the
form~\eqref{eq:jweaken2} in the free \textsc{gat} on $X$ comprises a
type judgement as above with associated element $(\phi, (\alpha, h))$,
together with a term expression $t$ such that either:
\begin{enumerate}[(i)]
\item $  x_{j_1} : T_{j_1}, \dots, x_{j_{k-1}} : T_{j_{k-1}} \vdash t
  : T_n$ is derivable in the free $\{s\}$-\textsc{gat} on $X$; or
\item $t = x_i$ for some $i \in [n-1]$ such that $T_n = T_i$; i.e.,
  such that $\phi(n) = \phi(i)$, $\alpha_n = \alpha_i$ and $h_n = h_i$.
\end{enumerate}
We may thus assign to $[\J'] \in \Tm{TX}{[\J]}$ an element of
$\Tm{PSX}{\phi, (\alpha, h)}$, lying in the left- or right-hand
summand according as $t$ is of the form (i) or (ii). An
inductive argument like the one above now shows that this mapping is
inverse to the component $\Tm {PSX}{\phi, (\alpha, h)} \to \Tm
{TX}{\kappa_X(\phi, (\alpha, h))}$ of $\kappa_X$.
\end{proof}
By transporting the monad structure of $T$ along the
isomorphisms~\eqref{eq:compositeendo}, we thus obtain a monad
structure on $PS$. With respect to this structure, the maps
$\eta^P S \colon S \Rightarrow PS$ and $S \eta^P \colon P \Rightarrow
PS$ now become monad morphisms which in addition satisfy the ``middle unit law'' expressed by the
commutativity of:
\[
\cd[@!C]{
PS \ar@{=>}[rr]^{P \eta^S \eta^P S} \ar@{=>}[dr]_{1} & & PSPS \ar@{=>}[dl]^{\mu^{PS}} \\
& PS\rlap{ .}
}
\]
Henceforth, we shall take it that in fact $T = PS$.

\subsection{Unit and multiplication}
We now describe the unit and the multiplication of the monad for
\textsc{gat}s in terms of those for the weakening-and-projection and
substitution monads. The case of the unit is straightforward.
 \begin{Prop}
   For each $X \in [\H^\op, \cat{Set}]$, the unit map $\eta^{PS}_X \colon X
   \to PSX$ of the monad for \textsc{gat}s is the composite
\begin{align*}
X \xrightarrow{\eta^S_X} SX \xrightarrow{\eta^P_{SX}} PSX
\end{align*}
\end{Prop}
\begin{proof}
  An immediate consequence of the fact that $\eta^P S$ is a monad map.
\end{proof}

The multiplication $\mu^{PS}$ of the monad for \textsc{gat}s may be
described in terms of those of $S$ and $P$ together with one
additional datum: that of a \emph{distributive law} of $S$ over $P$ in
the sense of~\cite{Beck1969Distributive}. This is a natural
transformation $\delta \colon SP \Rightarrow PS$ satisfying four
axioms relating it to the units and multiplications of the monads $S$
and $P$. It may be obtained from  the multiplication $\mu^{PS}$ 
as the composite:
\begin{equation}
  \label{eq:5}
\delta = SP \xRightarrow{\eta^PSP\eta^S} PSPS \xRightarrow{\mu^{PS}} PS\rlap{ .}
\end{equation}
In a moment, we shall give an explicit
description of $\delta$, but first let us record how it allows us to
reconstruct the multiplication of $PS$:
\begin{Prop}
  For each $X \in [\H^\op, \cat{Set}]$, the multiplication 
  $\mu^{PS}_X \colon PSPSX \to PSX$ of the monad for \textsc{gat}s is the
  composite
\[
PSPSX \xrightarrow{P\delta_{SX}} PPSSX \xrightarrow{\mu^P \mu^S_X}
PSX\rlap{ .}
\]
\end{Prop}
\begin{proof}
  This is (1) $\Leftrightarrow$ (2) of~\cite[Proposition,
  Section~1]{Beck1969Distributive}.
\end{proof}
Since we already have explicit descriptions of $\mu^P$ and $\mu^S$,
this allows us to reduce the problem of giving an explicit description
of $\mu^{PS}$ to that of giving one for $\delta$.  Such a description
is essentially an account of how the process of substituting terms
into a weakened judgement may be re-expressed as the process of
weakening a judgement to which substitution has already been
applied. The behaviour is different depending on whether the terms we
are substituting in are genuine terms or are projections onto a
variable; those of the former kind induce actual substitutions, while
those of the latter express the structural rule of \emph{contraction}.
Our description of $\delta_X$ will thus come in two parts, the first
dealing only with actual substitutions, and the second reintroducing
contraction.

First we need an explicit description of $SPX$.  Given $\alpha \in
\mathrm{Inc}(n)$, a map $[\alpha] \to PX$ is determined as in the
discussion following Definition~\ref{def:inclist} by its value at
$\alpha(n)$ and its values at $i_t$ for each $i \in \ab{\alpha(n)}
\setminus \im \alpha$, thus by giving:
\begin{itemize}
\item A pair $(\phi \in \mathrm{Hp}(\alpha(n)), h \colon \ab \phi \to X) \in
PX(\alpha(n)) = WX(\alpha(n))$; and
\item Elements $k(i) \in \Tm {PX}{\res \phi i, \res h i}$ for each
  $i \in [\alpha(n)] \setminus \im \alpha$,
\end{itemize}
and so we write a typical element of $SPX(n)$ as $(\alpha, (\phi, h,
k))$. By Proposition~\ref{prop:projendo}, the set $\Tm
{PX}{\res \phi i, \res h i}$ which each $k(i)$ inhabits is
the disjoint union
\[
\Tm X {h(i)} + \{\pi_j : j \in [i-1], \phi(i) = \phi(j), h(i) =
h(j)\}\rlap{ ;}
\] 
we will call $(\alpha, (\phi, h, k))$ \emph{projection-free} if
each $k(i)$ lies in the left-hand summand.
We first describe the action of $\delta_X$ on projection-free elements.
\begin{Defn}\label{def:projfree}
  Let $(\alpha, (\phi, h, k)) \in SPX(n)$ be projection-free.  We
  define the heap $\alpha^\ast\phi \in \mathrm{Hp}(n)$ in relation
  form by
$i \hp_{\alpha^\ast\phi} j$ iff $\alpha(i) \hp_\phi
\alpha(j)$.
Given $p \in [n]$ with $\dn_{\alpha^\ast \phi}(p) = \{p_1 \prec \dots
\prec p_m = p\}$, we define $\alpha^\phi_p \in \mathrm{Inc}(m)$ by $
\alpha^\phi_p(i) = \mathsf{dp}_\phi(\alpha(p_i))$.  If now
$\dn_\phi(\alpha(p)) = \{v_1 \prec \dots \prec v_{\ell} =
\alpha(p)\}$, then it is easy to see that $i \in [\ell]$ is in the
image of $\alpha^\phi_p$ just when $v_i$ is in the image of $\alpha$;
it follows that we have a well-defined mapping
$(h+k)_p \colon [\alpha^\phi_p] \to X$ given by
\[
(h+k)_p(i)= 
h(v_i) \quad \text{and} \quad (h+k)_p(i_t) = k(v_i)
\]
for $i \in [\ell]$ (on the left) and $i \in [\ell] \setminus \im
\alpha^\phi_p$ (on the right).  It is moreover easy to verify that
$\partial(\alpha^\phi_p, (h+k)_p) = (\alpha^\phi_q, (h+k)_q)$ whenever
$\alpha^\ast \phi(p) = q$, so that we have a well-defined mapping
$(\alpha^\phi, h + k) \colon \ab {\alpha^\ast \phi} \to SX$, and so an
element $(\alpha^\ast \phi, (\alpha^\phi, h+k)) \in PSX(n)$.
\end{Defn}

\begin{Ex}
If $\alpha$ is the inc-list $0 < 1 < 5 < 6 < 7 < 8$ and
$\phi$ is as on the left below, then $\alpha^\ast \phi$ and
$\alpha^\phi$ are as on the right.
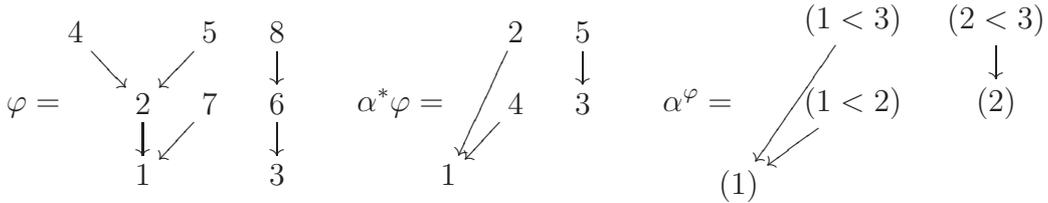
\[
  \phi = \!\! \cd[@-1em]{
4 \ar[dr] & & 5 \ar[dl] & 8 \ar[d] \\ &
2 \ar[d] & 7 \ar[dl] & 6 \ar[d]\\ & 
1 & & 3
} \qquad
\alpha^\ast\phi =\!\!\!\! \cd[@-1em]{
& 2 \ar[ddl] & 5 \ar[d] \\ 
 & 4 \ar[dl] & 3\\ 
1
} \qquad
\alpha^\phi = \!\!\!\!\!\!\cd[@-1em@C-0.2em]{
& (1<3) \ar[ddl] & (2<3) \ar[d] \\ 
 & (1<2) \ar[dl] & (2) \\ 
(1)
}
\]  
\end{Ex}

We now wish to describe the action of $\delta_X$ on arbitrary
type-elements. As a first step, let us call $(\alpha, (\phi, h, k))
\in SPX(n)$ \emph{nearly projection-free} if the only terms $k(i)$
which are projections are ones for which $i$ is not in the image of
$\phi$ (thus leaves in the forest corresponding to $\phi$). For such
an element, we can still define $(\alpha^\ast \phi, (\alpha^\phi,
h+k))$ as above; the point which requires checking is that, for $p \in
[n]$ with $\dn_\phi(\alpha(p)) = \{v_1 \prec \dots \prec v_{\ell} =
\alpha(p)\}$ and $i \in [\ell] \setminus \im \alpha^\phi_p$, the
element $k(v_i)$ should be a term of $X$ rather than a projection. But
this is true since we necessarily have $i < \ell$, so that $v_i$ is in
the image of $\phi$ and thus a term of $X$ by assumption.  With this
observation in mind, we may now extend our description of the action
of $\delta_X$ to deal with arbitrary type-elements.
\begin{Defn}\label{def:termfree}
  Given a general element $(\alpha, (\phi, h, k)) \in SPX(n)$, 
let $\leq_k$ be the partial
  order generated on $[0,\alpha(n)]$ by the basic inequalities:
\[
\text{$i \leq_k j$ when $j \notin \im \alpha$ and $k(j) = \pi_i$ ,}\]
and let $\bar \phi \in \mathrm{Hp}(\alpha(n))$ be given by $\bar
\phi(i) = \min\{j : j \leq_k \phi(i)\}$. Note that $i \leq_k j$
implies $\phi(i) = \phi(j)$ and $h(i) = h(j)$, which means that
$\mathsf{dp}_\phi(i) = \mathsf{dp}_{\bar \phi}(i)$ for all $i \in
[\alpha(n)]$ and, if $\bar \phi(i) > 0$, that $\partial h(i) = h(\bar
\phi(i))$. Thus $(\alpha, (\bar \phi, h, k))$ is a well-defined
element of $SPX(n)$ which is easily seen to be nearly projection-free, so
that we may form $(\alpha^\ast {\bar \phi},
(\alpha^{\bar \phi}, h + k)) \in PSX(n)$.

\end{Defn}

This completes our description of the action of $\delta_X$ on
type-elements; we now prove its validity at the same time as giving
the action on term-elements. For the latter, let us note that we
have $\Tm {SPX}{\alpha, (\phi, h, k)} = \Tm {PX}{\phi, h}$ which
by Proposition~\ref{prop:projendo} again is the disjoint union
\[
\Tm X {h(\alpha(n))} + \{\pi_j : j \in
[\alpha(n) - 1], \phi(\alpha(n)) = \phi(j), h(\alpha(n)) =
h(j)\}\rlap{ .}\]
 \begin{Prop}\label{prop:distrib}
   For each $X \in [\H^\op, \cat{Set}]$, the action of the
   distributive law $\delta_X$ is given on type-elements $SPX(n) \to PSX(n)$
   by
\begin{align*}
(\alpha, (\phi, h,k)) & \mapsto (\alpha^\ast\bar\phi, (\alpha^{\bar\phi},
h+k)) \rlap{ ,}
\end{align*}
and on term-elements
 $\Tm {SPX} {\alpha, (\phi, h, k)} \to \Tm {PSX} {\alpha^\ast
    \bar \phi, (\alpha^{\phi}, h + k)}$ by
\[
a \mapsto
\begin{cases}
a & \text{if $a \in \Tm X {h(\alpha(n))}$;} \\
\pi_m & \text{if $a = \pi_j$ and $\min\{i : i \leq_k j\} = \alpha(m)$;}\\
k(m) & \text{if $a = \pi_j$ and $m = \min\{i : i \leq_k j\} \notin \im
    \alpha$.}
\end{cases}
\]
\end{Prop}
\begin{proof} We prove the result for types and terms simultaneously
  by induction on $\alpha(n) - n$. For the base case $\alpha(n) = n$,
  we must have $\alpha = \iota_n$, and now on type-elements we have
  $\delta_X(\iota_n, (\phi, h, k)) = \delta_X(\eta^S_{PX}(\phi, h)) =
  P\eta^S_X(\phi, h) = (\phi, (\iota, \tilde h))$, which by direct
  calculation from the definitions is equal to $(\iota_n^\ast {\bar\phi},
  ({\iota_n}^{\bar\phi}, h+k))$. The argument for terms in the base
  case is similarly straightforward, on observing that $\leq_k$ in
  this case satisfies $i \leq_k j$ iff $i = j$.

  Before giving the inductive step, we make an observation.  Suppose
  given $(\phi, (\alpha, h)) \in PSX(n)$ and $t \in \Tm {PSX} {\res
    \phi m, \res{(\alpha, h)}{m}}$. We wish to describe the judgement
  obtained by substituting $\J(\res \phi m, \res{(\alpha, h)}{m}, t)$
  into $\J(\phi, (\alpha, h))$ in the free \textsc{gat} on $X$. If $t
  = \pi_j$ is a projection term, then direct inspection of the
  bijection of Proposition~\ref{eq:compositeendo} shows that the
  judgement obtained is $\J(\phi' \setminus m, (\alpha, h)\setminus
  m)$, where $\phi' \in \mathrm{Hp}(n)$ is defined by $\phi'(i) = j$
  if $\phi(i) = m$ and $\phi'(i) = \phi(i)$ otherwise. On the other
  hand, if $t = a \in \Tm{X}{h_m(\alpha_m(\# m))}$, inspection of
  Proposition~\ref{eq:compositeendo} shows that this substitution
  yields $\J(\phi' \!\setminus\! m, (\alpha', h')\!\setminus\! m)$,
  with $\phi' \in \mathrm{Hp}(n)$ and $(\alpha', h') \colon \ab{\phi'}
  \to SX$ given by
\begin{equation*}
\phi'(i) = \begin{cases}
\phi^2(i) & \text{if }\phi(i) = m\rlap{;}\\
\phi(i) & \text{otherwise; }\\
\end{cases}
\quad
(\alpha'_i, h'_i) = \begin{cases}
(\alpha_i \epsilon_\ell, h_i \cup a) & \text{if }\phi^{\# i - \ell}(i) = m\rlap{;}\\
(\alpha_i, h_i) & \text{otherwise.}
\end{cases}
\end{equation*}
Here we write $\epsilon_\ell \colon [0, n] \to [0, n+1]$ for the
unique injection whose image does not include $\ell$, and, as in the
proof of Proposition~\ref{prop:substendo}, write $h_i \cup a \colon
\ab{\alpha_i \epsilon_\ell} \to X$ for the map which extends $h_i
\colon \ab {\alpha_i} \to X$ by sending $(\alpha_i(\ell))_t$ to $a$.

We now give the inductive step of our main argument. Let $(\alpha,
\ell) = (\alpha, (\phi, h, k)) \in SPX(n)$ with $\alpha(n) > n$.  As
in Proposition~\ref{prop:substendo}, we can find $\alpha(m) < j <
\alpha(m+1)$ and form the type-elements $(\alpha_j, \ell_j) $ and
$(\alpha^j, \ell^j)$ and term-element $k(j) \in \Tm {SPX} {\alpha^j,
  \ell^j}$.  Now we have that $\ell_j = (\phi, h, k') \colon
\ab{\alpha_j} \to X$, where $k'$ is obtained from $k$ by removing the
value at $j$; and we have that $\ell^j = \partial^{n-m-1}(\ell_j)$. So
by induction, applying $\delta_X$ to $(\alpha_j, \ell_j) $ and
$(\alpha^j, \ell^j, k(j))$ yields the elements
  \begin{gather*}
    (\alpha_j{}^*\bar\phi, (\alpha_j{}^{\bar\phi}, h + k')) \quad 
\text{and } \quad
    (\res{\alpha_j{}^*\bar\phi}{m+1}, \res{(\alpha_j{}^{\bar\phi}, h + k')}{m+1},
    k(j))
  \end{gather*}
  of $PSX$.  The judgement $\J(\alpha, \ell)$ of the free
  $\{s\}$-\textsc{gat} on $PX$ is obtained by substituting
  $\J(\alpha^j, \ell^j, k(j))$ into $\J(\alpha_j, \ell_j)$, whence
  $\delta_X(\alpha, \ell)$ is obtained by substituting
  $\J(\res{\alpha_j{}^*\phi}{m+1}, \res{(\alpha_j{}^\phi, h +
    k')}{m+1}, k(j))$ into $\J(\alpha_j{}^*\phi, (\alpha_j{}^\phi, h +
  k'))$ in the free \textsc{gat} on $X$. Now $k(j)$ is either a
  projection or non-projection; applying the appropriate part of the
  above observation and calculating shows that, in either case, the
  resultant judgement is $\J(\alpha^\ast \bar\phi, \alpha^{\bar\phi},
  h+k)$, so that $\delta_X(\alpha, \ell) = (\alpha^\ast \bar\phi,
  \alpha^{\bar\phi}, h + k)$.

  Finally, we give the inductive step on term-elements. The key point
  is for us to extend the observation made above. Given $(\phi,
  (\alpha, h)) \in PSX(n)$ and $t \in \Tm{PSX}{\res \phi m, \res
    {(\alpha, h)} m}$ as before and also $a \in \Tm{PSX}{\phi,
    (\alpha, h)}$, we wish to describe the result of substituting
  $\J(\res \phi m, \res {(\alpha, h)} m, t)$ into $\J(\phi, (\alpha,
  h), a)$. If $v$ denotes the term-element representing this
  judgement, then direct calculation shows that
\begin{align*}
v = \begin{cases}
a & \text{if $a \in \Tm X {h_n(\alpha_n(\# n))}$;} \\
\pi_j & \text{if $a = \pi_j$ and $j < m$;}\\
t & \text{if $a = \pi_m$;}\\
\pi_{j-1} & \text{if $a = \pi_j$ and $j > m$;}
\end{cases}
\end{align*}
Applying this observation together with induction and the preservation
of derivations by $\delta_X$ now yields the inductive step on
terms. The details are similar to the type-element case and so omitted.
\end{proof}

Drawing together the results of the previous four sections, we thus
obtain the main result of the paper, giving a complete
characterisation of the monad for \textsc{gat}s on the presheaf
category $[\H^\op, \cat{Cat}]$.

\begin{Thm}
  The monad for \textsc{gat}s induced by the free-forgetful adjunction
  $\cat{GAT} \leftrightarrows [\H^\op, \cat{Set}]$ may be taken to
  have underlying endofunctor $PS$, where $P$ and $S$ are as in
  Propositions~\ref{prop:projendo} and \ref{prop:substendo}; 
  unit map at $X \in [\H^\op, \cat{Set}]$ given by $\eta^P
  \eta^S_X \colon X \to PSX$, where $\eta^P$ and $\eta^S$ are as in
  Propositions~\ref{prop:projunit} and~\ref{prop:substunit}; and 
  multiplication map at $X \in [\H^\op, \cat{Set}]$ given by the composite
\[
PSPSX \xrightarrow{P\delta_{SX}} PPSSX \xrightarrow{\mu^P \mu^S_X}
PSX\rlap{ ,}
\]
where $\mu^P$, $\mu^S$ and $\delta$ are as in
Propositions~\ref{prop:projmult}, \ref{prop:substmult},
and~\ref{prop:distrib} respectively.
\end{Thm}

\section{Categorical analysis}\label{sec:final}
We have now completed the main task of the paper by describing the
monads for $D$-\textsc{gat}s for $D = \{w\}$, $\{w, p\}$, $\{s\}$ and
$\{w,p,s\}$. The purpose of this final section is to discuss the good
categorical properties that these monads have: namely, those of being
\emph{local right adjoint} and \emph{cartesian}. These properties
justify us in regarding these monads as fundamentally combinatorial in
nature, and will allow us, in future work, to bring a rich body of
theory~\cite{Weber2004Generic,Leinster2004Operads,Weber2007Familial,Berger2012Monads}
to bear on the study of dependent sequent calculi. Let us begin by
briefly sketching some of these applications:
\begin{itemize}
\item \textbf{Nerve functors}. Weber's ``nerve
  theorem''~\cite{Weber2007Familial} allows us to associate to any
  local right adjoint, cartesian monad on a presheaf category a
  \emph{nerve functor} $T\text-\cat{Alg} \to [\E^\op, \cat{Cat}]$: a
  fully faithful embedding of the category of $T$-algebras into a
  presheaf category, together with a characterisation of the essential
  image of this functor. The importance of this is that a nerve
  functor can allow algebraic entities to be embedded into a geometric
  or topological context; hence this will allow us to explore
  geometric and higher-dimensional aspects of dependent sequent calculi.
\vskip0.5\baselineskip
\item \textbf{Categorical algebras}. Any local right adjoint monad $T
  \colon \C \to \C$ preserves pullbacks, and so lifts to a $2$-monad
  on the $2$-category $\cat{Cat}(\C)$ of categories internal to $\C$;
  an algebra for this lifted monad may be called a \emph{categorical
    $T$-algebra}. In particular, this means that we can consider
  ``categorical $D$-\textsc{gat}s''. The value of this is in allowing
  a new approach to the coherence problem
  of~\cite{Hofmann1995On-the-interpretation}, that many naturally
  occurring models of dependent type theory are ``too weak'' to be
  strict models of the syntax. This is resolved by observing that
  these models are actually \emph{categorical pseudoalgebras} for the
  lifted $2$-monad. By considering pseudomorphisms of algebras, we may
  perfectly well interpret the strict syntax in these weak models,
  thereby avoiding the use of strictifiction theorems. Among the
  categorical pseudoalgebras, we also find objects which represent the
  ``syntax with substitution up to isomorphism''
  of~\cite{Curien1993Substitution}; and now the two-dimensional monad
  theory of~\cite{Blackwell1989Two-dimensional} describes the relation
  between the strict and the weak syntax.\vskip0.5\baselineskip
\item \textbf{Lax morphisms}. As is well known, lax monoidal functors
  $1 \to \V$ from the terminal monoidal category classify monoids in
  $\V$. In a similar way, if $\mathbb T$ is a categorical
  $D$-\textsc{gat}, then lax morphisms of $D$-\textsc{gat}s $1 \to
  \mathbb T$ correspond to models of type theory \emph{internal} to
  $\mathbb T$. In particular, one may generate the \emph{free
    categorical $D$-\textsc{gat} containing a $D$-\textsc{gat}}. This
  should be a fundamental combinatorial object, by analogy with the
  case of monoidal categories, where the corresponding entity is
  $\Delta_+$, the category of finite ordinals and monotone
  maps.
\end{itemize}
Investigating these ideas fully will be a paper in itself; for now, we
merely show that the monads under investigation are indeed local
right adjoint and cartesian.
\subsection{Local right adjoint and (strongly) cartesian monads}
\label{sec:local-right-adjoint}
We begin by revising the
notions of interest;
see~\cite{Carboni1995Connected,Leinster2004Higher,Weber2004Generic,Weber2007Familial,Berger2012Monads}
for further discussion and applications.

\begin{Defn}
  A functor $F \colon \C \to \D$ is called \emph{local right adjoint}
  if, for each $X \in \C$, the functor $F / X \colon \C / X \to \D /
  FX$ on slice categories is a right adjoint. A monad $T$ is called
  \emph{local right adjoint} when its underlying endofunctor is so.
\end{Defn}
By standard pasting properties of pullbacks, if $\C$ has a terminal
object then a functor $F \colon \C \to \D$ is local right adjoint just
when $F / 1 \colon \C / 1 \to \D / F1$ is a right adjoint. Such an $F$
is thus determined by its value $F1$ at the terminal object together
with a functor $G_1 \colon \D / F1 \to \C$ left adjoint to $F/1$.  In
the case $\C = \D = [\H^\op, \cat{Set}]$ of interest to us, a standard
categorical argument shows that giving the left adjoint $G_1$ is
equivalent to giving an arbitrary functor $[\thg] \colon
\mathrm{el}(F1) \to [\H^\op, \cat{Set}]$. Here $\mathrm{el}(F1)$ is
the \emph{category of elements} of $F1$, whose object set is
$\Sigma_{h \in \H} F1(h)$, and whose morphisms $(x \in F1(h)) \to (x'
\in F1(h'))$ are maps $f \colon h \to h'$ such that $x =
(F1)(f)(x')$. Given $F1$ and $[\thg]$, we can reconstruct $F$ from it
by the formula
\begin{equation}\label{eq:freconstruct}
FX(h) \cong \sum_{x \in F1(h)} [\H^\op, \cat{Set}](\ab {x}, X)\rlap{ .}
\end{equation}
This expresses that an element of $FX(h)$ is an element $x$
of $F1(h)$ together with an appropriate labelling $\ab x \to X$ by
elements of $X$. Thus elements of $F1$ can be seen as encoding the
``shapes'' of the operations appearing in the functor $F$.

Note that the formula~\eqref{eq:freconstruct} expresses the functor
$F(\thg)(h) \colon [\H^\op, \cat{Set}] \to \cat{Set}$ as a coproduct
of representables for each $h \in \H$. This provides an alternative
characterisation of the local right adjoint endofunctors of $[\H^\op,
\cat{Set}]$, and we record this result as:
\begin{Prop}\label{prop:lra}
  For an endofunctor $F$ of $[\H^\op, \cat{Set}]$, the following are equivalent:
  \begin{enumerate}[(i)]
  \item $F$ is local right adjoint;
  \item $F / 1 \colon [\H^\op, \cat{Set}] \to [\H^\op, \cat{Set}]/ F1$ is a right adjoint;
  \item There are given $F1$ and $[\thg] \colon \mathrm{el}(F1) \to
    [\H^\op, \cat{Set}]$ such that~\eqref{eq:freconstruct} is validated;
  \item For each $h \in \H$, the functor $F(\thg)(h) \colon [\H^\op,
    \cat{Set}] \to
    \cat{Set}$ is a coproduct of representable functors;
  \item $F$ preserves connected limits (i.e., all small fibre products
    and equalisers).
  \end{enumerate}
\end{Prop}
[The only part we have not discussed is the equivalence of (iv) and
(v), which follows immediately from the fact that
the functors $[\H^\op, \cat{Set}] \to \cat{Set}$ which preserve
connected limits are precisely the coproducts of representables.]



\begin{Defn}
  A natural transformation $\alpha \colon F \Rightarrow G \colon \C
  \to \D$ is called \emph{cartesian} if all its naturality squares are
  pullbacks.  A monad $T$ is called \emph{cartesian} if it preserves
  pullbacks and its unit $1 \Rightarrow T$ and multiplication $ TT
  \Rightarrow T$ are cartesian. A monad is \emph{strongly cartesian}
  if it is cartesian and local right adjoint.
\end{Defn}
Again, if $\C$ has a terminal object, then this definition simplifies:
a natural transformation $\alpha$ as above is cartesian if and only if
each naturality square of the following form is a pullback:
    \begin{equation}\label{eq:naturalitycartesian}
    \cd{FX \ar[r]^{\alpha_X} \ar[d]_{F!} & GX \ar[d]^{G!} \\
      F1 \ar[r]_{\alpha_1} & G1\rlap{ .}}
  \end{equation}

\subsection{Categorical analysis}
We now consider the above notions in the context of the monads $W$,
$P$, $S$ and $T = PS$ for weakening, for weakening and projection, for
substitution, and for \textsc{gat}s. We will see that $W$, $P$ and $S$
are all strongly cartesian, but that $T$, though local right adjoint,
is \emph{not} strongly cartesian.

\begin{Prop}\label{prop:weakeningcart} The weakening monad $W$ is
  strongly cartesian.
\end{Prop}
\begin{proof}
  We first show $W$ is local right adjoint using the
  characterisation of Proposition~\ref{prop:lra}(iv). For each $n \in
  \H$, we have $W(\thg)(n) = \sum_{\phi \in \mathrm{Hp}(n)}\, [\H^\op,
  \cat{Set}](\ab \phi, \thg)$ a coproduct of representables as
  required. As for $W(\thg)(n_t)$, we define for each $\phi \in
  \mathrm{Hp}(n)$ a presheaf $[\phi]_t \in [\H^\op, \cat{Set}]$ by
  taking $[\phi]$ and adjoining a new term-element over $n \in
  [\phi](\# n)$. Now $(WX)(n_t) = \sum_{(\phi, h) \in PX(n)}\Tm X
  {h(n)} \cong \sum_{\phi \in \mathrm{Hp}(n)}[\H^\op,
  \cat{Set}]([\phi]_t, X)$, whence $W(\thg)(n_t)$ is a coproduct of
  representables as required.
  We next show that $\eta^W \colon 1 \Rightarrow W$ is cartesian, thus
  that each naturality square~\eqref{eq:naturalitycartesian} is a
  pullback.  Evaluating at $n \in \H$, this says that the left
  square below is a pullback; while doing so at $n_t \in \H$
  is the requirement that the right square be a pullback for each $A \in X(n)$.
\[
\cd{X(n) \ar[r]^-{\eta^W_X} \ar[d]_{!} & WX(n) \ar[d]^{W!} \\
1(n) \ar[r]_-{\eta^W_1} & W1(n)} \qquad \qquad
\cd{\Tm X A \ar[r]^-{\eta^W_X} \ar[d]_{!} & \Tm {WX} {\xi_n, \tilde A} \ar[d]^{W!} \\
\Tm 1 {\star} \ar[r]_-{\eta^W_Y} & \Tm {W1} {\xi_n}}\rlap{ .}
\]
To say that the left square is a pullback is to say that each $(\gamma_n,
h) \in WX(n)$ is of the form $(\gamma_n, \tilde A)$ for a unique $A \in
X(n)$. But as we observed before Definition~\ref{prop:weakunit},
$\ab{\gamma_n}$ is the representable functor $\H(\thg, n)$, and so this
follows from the Yoneda lemma. For the right-hand square, we have
that $\Tm{WX}{\gamma_n, \tilde A} = \Tm{X}{\tilde A(n)} = \Tm X A$;
similarly $\Tm{W1}{\gamma_n} = \Tm{1}{\star}$, so that 
both horizonal maps are isomorphisms and the square is a pullback.
Finally, we show that $\mu^W \colon WW \Rightarrow W$ is cartesian,
thus that the squares:
\[
\cd{W^2X(n) \ar[r]^{\mu^W_X} \ar[d]_{W^2!} & WX(n) \ar[d]^{W!} \\
W^21(n) \ar[r]_{\mu^W_Y} & W1(n)} \qquad \qquad
\cd{\Tm {W^2X} {\psi, (\phi, h)} \ar[r]^{\mu^W_X} \ar[d]_{W^2!} & \Tm
  {WX} {\psi \star \phi, \psi \star h} \ar[d]^{W!} \\
\Tm {W^21} {\psi, \phi} \ar[r]_{\mu^W_1} & \Tm {W1} {\psi \star \phi}}
\]
are pullbacks for all $n \in \H$ and all $(\psi, (\phi, h)) \in
W^2X(n)$.  For the left square, we must show
that for each $(\psi, \phi) \in W^21(n)$ and $(\psi \star \phi, k)
\in WX(n)$, there's a unique $(\psi, (\phi, h)) \in W^2X(n)$ with $k =
\psi \star h$. So for each $i \in [n]$, consider the set $\dn_\psi(i) =
\{0 \prec v_1 \prec \dots \prec v_m = i\}$, and now define $h_i \colon
\ab{\phi_i} \to X$ by $h_i(j) = k(v_j)$. By the
definition~\eqref{eq:relationstar} of $\psi \star \phi$, we see that,
if $\phi_i(j) = \ell$, then $(\psi \star \phi)(v_j) = v_\ell$, whence
$h_i(\ell) = k(v_\ell) = k((\psi \star \phi)(v_j)) = \partial (k(v_j))
= \partial(h_i(j))$; thus $h_i$ is a well-defined map. It is moreover
easy to see that $\partial(\phi_i, h_i) = (\phi_j, h_j)$ whenever
$\psi(i) = j$, so that $(\phi, h) \colon \ab \psi \to WX$ is
well-defined; finally, it it straightforward to check that $\psi \star
h = k$, and that $h$ is unique with this property. 
Finally, for the right square, we observe that $\Tm {W^2X}{\psi,
  (\phi, h)} = \Tm {WX}{h_n(\# n)} = \Tm X {(\psi \star h)(n)} = \Tm
{WX}{\psi \star \phi, \psi \star h}$. Thus both horizontal arrows are
isomorphisms and the square is a pullback.
\end{proof}

\begin{Prop}\label{prop:projcart} The weakening and projection monad $P$ is strongly cartesian.
\end{Prop}
\begin{proof}
  We first show that $P$ is local right adjoint. Certainly $P(\thg)(n)
  = W(\thg)(n)$ is a coproduct of representables; as for
  $P(\thg)(n_t)$, we have that
  \begin{equation*}
    PX(n_t) = WX(n_t) + 
    \textstyle\sum_{(\phi, h) \in PX(n)} \{\pi_i : i \in [n-1], \phi(n) = \phi(i), h(n) = h(i)\}\text{ ,}
  \end{equation*}
  so by Proposition~\ref{prop:weakeningcart}, it suffices to exhibit
  the right summand as a coproduct of representables.  A typical
  element $(\phi, h, \pi_i)$ of this set determines and is determined
  by a triple $(\partial \phi \in \mathrm{Hp}(n-1), \partial h \colon
  \ab{\partial \phi} \to X, i \in [n-1])$ subject to no further
  conditions, so that this summand may be written as $ \sum_{{\phi \in
      \mathrm{Hp}(n-1), i \in [n-1]}} [\H^\op, \cat{Set}]([\phi], X) $
  as required.  It remains to show that $\eta^P$ and $\mu^P$ are
  cartesian. Since $P$ agrees with $W$ on type-elements, the only
  extra work involves term-elements: we must show that squares of the
  form
\[
\cd{\Tm X A \ar[r]^-{\eta^P_X} \ar[d]_{!} & \Tm {PX} {\gamma_n, \tilde A} \ar[d]^{P!} \\
\Tm 1 {\star} \ar[r]_-{\eta^P_Y} & \Tm {P1} {\gamma_n}} \qquad 
\cd{\Tm {P^2X} {\psi, (\phi, h)} \ar[r]^{\mu^P_X} \ar[d]_{P^2!} & \Tm
  {PX} {\psi \star \phi, \psi \star h} \ar[d]^{P!} \\
\Tm {P^21} {\psi, \phi} \ar[r]_{\mu^P_1} & \Tm {P1} {\psi \star \phi}}
\]
are pullbacks for all $A \in X(n)$ and for all $(\psi, (\phi, h)) \in
P^2X(n)$. For the left-hand square, note that $\Tm{PX}{\gamma_n, \tilde
  A}$ contains no projection terms $\pi_i$, as $\gamma_n(n) \neq
\gamma_n(i)$ for any $i \in [n-1]$. Thus $\Tm{PX}{\gamma_n, \tilde A} =
\Tm{WX}{\gamma_n, \tilde A}$ and similarly for $P1$, and so we may appeal
to Proposition~\ref{prop:weakeningcart}. Finally, for the right-hand
square, we need only deal with the new projection terms. We must show two
things:
\begin{itemize}
\item Given projection terms $\pi_i(\psi, \phi) \in P^21(\psi, \phi)$
  and $\pi_i \in PX(\psi \star \phi, \psi \star h)$, we have a valid
  projection term $\pi_i(\psi, (\phi, h)) \in P^2X$; if this
  exists, it will clearly be the unique element sitting over $\pi_i(\psi,
  \phi)$ and $\pi_i$.  Since $\pi_i(\psi, \phi)$ is a projection term,
  we already have that $\psi(n) = \psi(i)$ and $\phi_n = \phi_i$; and
  so we need only show that also $h_n = h_i$. Since $\psi(n) =
  \psi(i)$, we have $\partial h_n = \partial
  h_i$; it remains to show that $h_i$ and $h_n$ agree at $\# n = \# i$. But since
  $\pi_i \in PX(\psi \star \phi, \psi \star h)$, we have $(\psi \star
  h)(n) = (\psi \star h)(i)$, so by definition 
  $h_n(\# n) = h_i(\# i)$ as required.\vskip0.5\baselineskip
\item Given projection terms $\pi_i(\phi_n) \in P^21(\psi, \phi)$ and
  $\pi_{\psi^{\# n - i}(n)} \in PX(\psi \star \phi, \psi \star h)$, we
  have a valid projection term $\pi_i(\phi_n, h_n) \in P^2X$. Since
  $\pi_i(\phi_n)$ is a projection term, we already have that
  $\phi_n(\# n) = \phi_n(i)$, so it remains to show that $h_n(\# n) =
  h_n(i)$. Let us write $j = \psi^{\# n - i}(n)$. Since $\pi_j \in
  PX(\psi \star \phi, \psi \star h)$, we have $(\psi \star h)(n) =
  (\psi \star h)(j)$, whence $h_n(\#n) = (\psi \star h)(n) = (\psi
  \star h)(j) = h_j(\# j) = h_j(i) = h_n(i)$, as required, where for
  the last step, we use the fact that $j \in \dn_\psi(n)$ and so that
  $h_j = \res{(h_n)}{i}$.\qedhere
\end{itemize}
\end{proof}
\begin{Prop}\label{prop:substcart} The substitution monad $S$ is
  strongly cartesian.
\end{Prop}
\begin{proof}
  We first show that $S$ is local right adjoint. Arguing as in
  Proposition~\ref{prop:weakeningcart}, we have that $S(\thg)(n) =
  \sum_{\alpha \in \mathrm{Inc}(n)}\, [\H^\op, \cat{Set}](\ab \phi,
  \thg)$ is a coproduct of representables, while we may write
  $S(\thg)(n_t)$ as the coproduct $\sum_{\phi \in
    \mathrm{Inc}(n)}[\H^\op, \cat{Set}]([\alpha]_t, X)$, where the
  presheaf $[\alpha]_t \in [\H^\op, \cat{Set}]$ is obtained by
  adjoining to $[\alpha]$ a new term-element over $\alpha(n) \in
  [\phi](\alpha(n))$.  We next show that $\eta^S \colon 1 \Rightarrow
  S$ is cartesian; which, as before, is to show that the squares:
\[
\cd{X(n) \ar[r]^-{\eta^S_X} \ar[d]_{!} & SX(n) \ar[d]^{S!} \\
1(n) \ar[r]_-{\eta^S_1} & S1(n)} \qquad \qquad
\cd{\Tm X A \ar[r]^-{\eta^S_X} \ar[d]_{!} & \Tm {SX} {\iota_n, \tilde A} \ar[d]^{S!} \\
\Tm 1 {\star} \ar[r]_-{\eta^S_Y} & \Tm {S1} {\iota_n}}\rlap{ .}
\]
are pullbacks for each $n \in \H$ and each $A \in X(n)$. The argument
is exactly as in Proposition~\ref{prop:weakeningcart}, using the facts that
$\ab{\iota_n}$ is again the representable $\H(n, \thg)$, and 
that $\Tm{SX}{\iota_n, \tilde A} = \Tm{X}{\tilde A(\iota_n(n))} = \Tm X A$.
Finally, we show that $\mu^S \colon SS \Rightarrow S$ is cartesian,
thus that the squares:
\[
\cd{S^2X(n) \ar[r]^{\mu^S_X} \ar[d]_{S^2!} & SX(n) \ar[d]^{S!} \\
S^21(n) \ar[r]_{\mu^S_Y} & S1(n)} \qquad \qquad
\cd{\Tm {S^2X} {\alpha, (\beta, h, k)} \ar[r]^{\mu^S_X} \ar[d]_{S^2!} & \Tm
  {SX} {\beta\alpha, h \cup k} \ar[d]^{S!} \\
\Tm {S^21} {\alpha, \beta} \ar[r]_{\mu^S_1} & \Tm {S1} {\beta\alpha}}
\]
are pullbacks for all $n \in \H$ and all $(\alpha, (\beta, h, k)) \in
S^2X(n)$.  For the left square, we must show that for each $(\alpha,
\beta) \in S^21(n)$ and $(\beta\alpha, \ell) \in SX(n)$, there's a
unique $(\alpha, (\beta, h, k)) \in S^2X(n)$ with $\ell = h \cup
k$. But this is easy: we define $h(i) = \ell(i)$ for $i \in
[\beta(\alpha(n))]$ and $h(i_t) = \ell(i_t)$ for $i \in
[\beta(\alpha(n))] \setminus \im \beta$, and define $k(j) =
\ell(\beta(j)_t)$ for $j \in [\alpha(n)] \setminus \im \alpha$.  It is
easy to see that this is well-defined, that $h \cup k = \ell$, and
that $h$ and $k$ are unique with this property.  Finally, for the
right square, we observe that $\Tm {S^2X}{\alpha, (\beta, h, k)} = \Tm
{SX}{\beta, h} = \Tm X {\beta(\alpha(n))} = \Tm {SX}{\beta\alpha, h
    \cup k}$. Thus both horizontal arrows are isomorphisms and the
  square is a pullback.
\end{proof}
We conclude by considering the categorical properties of
the monad $T = PS$ for \textsc{gat}s. One might expect that $T$, like
its constituent monads $P$ and $S$, would be strongly
cartesian. However, this turns out not to be the case.
\begin{Prop}\label{prop:gatlra}
  The monad $T$ for \textsc{gat}s is local right adjoint, but not
  strongly cartesian.
\end{Prop}
\begin{proof}
  The underlying endofunctor $T = PS$ is the composite of two local
  right adjoint functors, and so itself is local right
  adjoint. Similarly, the unit $\eta^{PS} = \eta^P S \circ \eta^S
  \colon 1 \Rightarrow PS$ is the composite of two cartesian natural
  transformations and so cartesian. However, the same is \emph{not}
  true of the multiplication. The problem is that the distributive law
  $\delta \colon SP \Rightarrow PS$ is not a cartesian natural
  transformation; it follows that $\mu^{PS}$ is not either. Indeed, if
  it were, then both components of~\eqref{eq:5} would be cartesian
  (the first since $\eta^P$ and $\eta^S$ are and $PS$ preserves
  pullbacks), and hence the composite $\delta$ would be, too. It
  remains to prove that $\delta$ is not cartesian.  Let $X \in
  [\H^\op, \cat{Set}]$ be generated by a single element $A \in X(1)$,
  and consider the square on the left in:
\[
\cd{SPX(1) \ar[r]^{\delta_X} \ar[d]_{SP!} & PSX(1) \ar[d]^{PS!} \\
SP1(1) \ar[r]_{\delta_1} & PS1(1)}
\qquad \qquad
\cd{ & (\gamma_1, (\iota_1, \tilde A)) \ar@{|->}[d]^{} \\
(\alpha, \phi) \ar@{|->}[r]_{} & (\gamma_1, \iota_1)\rlap{ .}}
\]
We will show that this square is not a pullback. Let $\alpha \in
\Inc(1)$ be given by $\alpha(0) = 0$ and $\alpha(1) = 2$, and let
$\phi \in \mathrm{Hp}(\alpha(1)) = \mathrm{Hp}(2)$ be given by
$\phi(2) = \phi(1) = \phi(0) = 0$. These data give rise to a
projection-free term $(\alpha, \phi)$ in $SP1(1)$, and easily
$\alpha^\ast \phi = \gamma_1$ and $\alpha^\phi = \iota_1$. Thus we
have a diagram of elements as on the right above; but there can be no
element of $SPX(1)$ forming a cone over it. For indeed, such an element would
have to be a projection-free element $(\alpha, (\phi, h, k))$, where
$h \colon \ab \phi \to X$ and $k$ comprises an element $k(1) \in \Tm X
{h(1)}$; but since $X$ has no term-elements, this is impossible.
\end{proof}
\begin{Rk}
  It is probably worth explaining what the above failure of
  cartesianness means in proof-theoretic terms. Consider the following
  pattern of derivation. Take a type judgement $\vdash A \ty$. Weaken it
  with respect to a closed type $B$ to obtain $x: B \vdash A\
  \ty$. Now substitute in a closed term $\vdash t : B$. The result is,
  of course, once again just $\vdash A \ty$. The basic data for this
  derivation---the types $A$ and $B$ and the term $t$---correspond to
  an element $(\alpha, (\phi, h, k)) \in SPX(1)$, where $\alpha$ and
  $\phi$ are as in the preceding proof. The result of the
  derivation---the judgement $\vdash A \ty$---is the resultant
  element $\delta_X(\alpha, (\phi, h, k)) = (\gamma_1, (\iota_1,
  \tilde A)) \in PSX(1)$. When $X = 1$, this reduces to
  $\delta_1(\alpha, \phi) = (\gamma_1, \iota_1)$ and in fact $(\alpha,
  \phi)$ is the \emph{unique} preimage of $(\gamma_1, \iota_1)$ under
  $\delta_1$. Thus if $\delta$ were to be cartesian, then each
  $(\gamma_1, (\iota_1, \tilde A))$ would also have to have a unique
  preimage of the form $(\alpha, (\phi, \thg, \thg))$. But this would be to
  say that there were a unique derivation of $\vdash A \ty$ given by
  weakening  and then substituting as above. This is clearly is not
  so: there is one such derivation for each closed type $B$ and each
  term $t$ of that type.

  Thus, finally, the reason for the failure of cartesianness is that,
  in the presence of weakening, application of substitution may
  destroy information like that of $B$ and $t$ in the above example.
  Thus the failure is a failure of \emph{linearity}. What is perhaps
  remarkable is that substitution \emph{is} linear in this sense with
  respect to the theory without weakening; it is only the interaction
  of substitution with weakening that causes the problem.
\end{Rk}

 \bibliographystyle{acm}

\bibliography{bibdata}
 
\end{document}